\documentclass[11pt]{amsart}
\usepackage{amsthm,amsmath,amssymb}
\usepackage{xypic}

\newcommand{\de}{\delta}
\newcommand{\ep}{\varepsilon}
\newcommand{\dblq}{{/\!/}}
\newcommand{\HH}{{\mathbb{H}}}
\newcommand{\EE}{{\mathbb{E}}}
\newcommand{\calO}{{\mathcal O}}
\newcommand{\calS}{{\mathcal S}}
\newcommand{\calH}{{\mathcal H}}
\newcommand{\calF}{{\mathcal F}}
\newcommand{\calI}{{\mathcal I}}
\newcommand{\calA}{{\mathcal A}}
\newcommand{\calC}{{\mathcal C}}
\newcommand{\calL}{{\mathcal L}}
\newcommand{\calM}{{\mathcal M}}
\newcommand{\calJ}{{\mathcal J}}

\newcommand{\calE}{{\mathcal E}}

\newcommand{\calX}{{\mathcal X}}

\newcommand{\calD}{{\mathcal D}}
\newcommand{\calR}{{\mathcal R}}
\newcommand{\calQ}{{\mathcal Q}}
\newcommand{\calU}{{\mathcal U}}
\newcommand{\calN}{{\mathcal N}}

\def\GG{{\textbf G}}
\def\PP{{\textbf P}}
\def\BB{{\textbf B}}

\newcommand{\op}{\operatorname}

\newcommand{\Sp}{\op{Sp}(2g,{\bf Z})}

\newcommand{\Sym}{\op{Sym}}

\newcommand{\rk}{\op{rk}}

\newcommand{\Pic}{\op{Pic}}
\newcommand{\Jac}{\op{Jac}}
\newcommand{\Ker}{\op{Ker}}

\newcommand{\even}{{\op{even}}}

\newcommand{\T}{\Theta}
\renewcommand{\t}{\theta}
\newcommand{\tn}{\t_{\rm null}}

\renewcommand\tt[2]{\t\left[\begin{matrix}#1\\ #2\end{matrix}\right]}

\newcommand{\p}{\partial}
\newcommand{\Sn}{\calS_{\rm null}}
\newcommand{\Sing}{\op{Sing}}

\def\aa{\overline{\mathcal{A}}}
\def\mm{\overline{\mathcal{M}}}
\def\rr{\overline{\mathcal{R}}}
\def\qq{\overline{\mathcal{Q}}}
\def\uu{\overline{\mathcal{U}}}

\def\pr{\widetilde{\mathcal{R}}}

\theoremstyle{plain}
\newtheorem{thm}{Theorem}[section]
\newtheorem{lm}[thm] {Lemma}
\newtheorem{prop}[thm]{Proposition}
\newtheorem{cor}[thm]{Corollary}

\theoremstyle{definition}
\newtheorem{df}[thm] {Definition}
\newtheorem{rem}[thm]{Remark}

\begin{document}
\title{Singularities of theta divisors and the geometry of $\calA_5$}
\author{G.  Farkas}
\address{Humboldt Universit\"at zu Berlin, Institut f\"ur Mathematik, Unter den Linden 6, 10099 Berlin, Germany}
\email{farkas@math.hu-berlin.de}
\author{S. Grushevsky}
\address{Stony Brook University, Department of Mathematics, Stony Brook, NY 11790-3651, USA}
\email{sam@math.sunysb.edu}
\author{R. Salvati Manni}
\address{Universit\`a ``La Sapienza'', Dipartimento di Matematica, Piazzale A. Moro 2, I-00185, Roma,   Italy}
\email{salvati@mat.uniroma1.it}
\author{A.  Verra}
\address{Universit\`a ``Roma 3'', Dipartimento di Matematica, Largo San Leonardo Murialdo I-00146 Roma, Italy}
\email{verra@mat.uniroma3.it}

\begin{abstract}
We study the codimension two locus $H$ in $\calA_g$  consisting of principally polarized abelian varieties whose theta divisor has a singularity {}that{} is not an ordinary double point. We compute the class $[H]\in CH^2(\calA_g)$ for every $g$. For $g=4$, this turns out to be the locus of Jacobians with a vanishing theta-null. For $g=5$, via the Prym map we show that $H\subset \calA_5$ has two components, both unirational, which we describe {}completely{}.
We then determine the slope of the effective cone of $\aa_5$ and show that the component $\overline{N_0'}$ of the Andreotti-Mayer divisor has minimal slope.
\end{abstract}
\maketitle

\section*{Introduction}

The theta divisor $\T$ of a generic principally polarized abelian variety (ppav) is smooth. The ppav $(A,\T)$ with a singular theta divisor form the Andreotti-Mayer divisor $N_0$
in the moduli space $\calA_g$, see \cite{anma} and \cite{beauville}. The divisor $N_0$ has two irreducible components, see \cite{mumforddimag} and
\cite{debarredecomposes}, which are denoted $\tn$ and $N_0'$: here $\tn$ denotes the locus of ppav for which the theta divisor has a singularity at a two-torsion
point, and $N_0'$ is the closure of the locus of ppav for which the theta divisor has a singularity not at a two-torsion point. The theta divisor $\T$  of a generic
ppav $(A,\T)\in\tn$ has a unique singular point, which is a double point. Similarly, the theta divisor of a generic element of $N_0'$ has two distinct double singular points
$x$ and $-x$. Using this fact, one can naturally assign multiplicities to both components of $N_0$ such the following equality of cycles holds, see \cite{mumforddimag},\cite{debarredecomposes}
\begin{equation}\label{divisor}
 N_0=\tn+2N_0 '.
\end{equation}
As it could be expected, generically for both components the double point is an ordinary double point (that is, the quadratic tangent cone to the theta divisor at such a point
has maximal rank  $g$ {}---{} equivalently, the Hessian matrix of the theta function at such a point is non-degenerate).
Motivated by a conjecture of H.~Farkas \cite{farkashvanishing}, in \cite{grsmgen4} two of the authors considered the locus in $\tn$ in genus 4 where the double point is not
ordinary. In \cite{grsmordertwo} this study was extended to arbitrary $g$, considering the sublocus $\tn^{g-1}\subset\tn$ parameterizing ppav $(A,\T)$ with a singularity at a two-torsion point, that is not an ordinary double point of $\T$. In particular it has been proved that
\begin{equation}\label{cap}
\tn^{g-1}\subset \tn\cap N_0'.
\end{equation}
In fact the approach yielded a more precise statement:  Let $\phi:\calX_g\to\calA_g$ be the universal family of ppav over the orbifold $\calA_g$ and $\calS\subset\calX_g$ be the locus of singular points of theta divisors. Note that $\calS$ can be viewed as a
subscheme of $\calX_g$ given by the vanishing of the theta functions and all its partial derivatives, see Section 1. Then $\calS$ decomposes into three equidimensional components \cite{debarredecomposes}: $\Sn$, projecting to $\tn$, $\calS'$, projecting to $N_0'$, and $\calS_{\rm dec}$, projecting (with $(g-2)$-dimensional fibers) onto $\calA_1\times\calA_{g-1}$. It is proved in \cite{grsmordertwo} that set-theoretically, $\tn^{g-1}$ is the image in $\calA_g$ of the intersection $\Sn\cap\calS'$.
An alternative proof of these results has been found by Smith and Varley \cite{smvagen4},\cite{smvageng}.

\smallskip
It is natural to investigate the non-ordinary double points on the other component $N_0'$ of the Andreotti-Mayer divisor.
Similarly to $\tn^{g-1}$, we define $N_0'^{g-1}$, or, to simplify notation,  $H$,
to be the closure in $N_0'$ of the locus of ppav whose theta divisor has a non-ordinary double point singularity.
Note that $H$ is the pushforward under $\phi$ of a subscheme $\mathcal{H}$ of $\calX_g$ given by a Hessian condition on theta functions. In particular $H$ can be viewed as a codimension $2$ cycle (with multiplicities) on $\calA_g$.
Since an explicit modular form defining $N_0'$ and the singular point is not known, we
consider the cycle
\begin{equation}\label{cup}
 N_0^{g-1}:=\tn^{g-1}+2N_0'^{g-1}=\tn^{g-1} +2H.
\end{equation}

We first note that $\tn^{g-1}$ is a subset of  $H$. Then, after recalling that the Andreotti-Mayer loci $N_i$ are defined as consisting of  ppav $(A,\T)\in \calA_g$ with $\mathrm{dim} \  \mathrm{Sing}(\T)\geq i $, we establish the set-theoretical inclusion $N_i\subset H$, for $i\geq 1$.
 From this we deduce:
\begin{prop}\label{prop:ne}
For $g\ge 5$ we have $\tn^{g-1}\subsetneq  H$.
\end{prop}
To further understand the situation, especially in low genus, we compute the class:

\begin{thm}\label{thm:class}
The class of the cycle  $H$ inside $\calA_g$ is equal to
$$
 [H]=\left( \frac{g!}{16}(g^3+ 7g^2+18g+24)-(g+4)2^{g-4}( 2^g+1)\right)\lambda_1^2\in CH^2(\calA_g).
$$
\end{thm}
As usual, $\lambda_1:=c_1(\EE)$ denotes the first Chern class of the Hodge bundle and   $CH^i$   denotes  the ${\bf Q}$-vector space  parameterizing algebraic cycles of codimension $i$ with rational coefficients modulo rational equivalence.
Comparing classes and considering the cycle-theoretic inclusion $3\tn^{3}\subset H$, we get the following result, see Section 4 for details:
\begin{thm}\label{thm:genus4}
In genus $4$ we have the set-theoretic equality
$\tn^{3}= H$.
\end{thm}

We then turn to genus $5$ with the aim of obtaining a  geometric description of $H\subset \calA_5$ via the {}dominant{}
Prym map $P:\calR_6\rightarrow \calA_5$. A key role in the study of the Prym map is played by its branch divisor, which in this case equals $N_0'\subset \calA_5$, and its ramification divisor $\calQ\subset \calR_6$. We introduce the \emph{antiramification} divisor $\calU\subset \calR_6$ defined cycle-theoretically by the equality $$P^*(N_0')=2\calQ+\calU.$$
Using the geometry of the Prym map, we describe both $\calQ$ and $\calU$ explicitly in terms of Prym-Brill-Noether theory. For a Prym curve $(C, \eta)\in \calR_g$ and an integer $r\geq -1$, we recall that $V_r(C, \eta)$ denotes the Prym-Brill-Noether locus (see Section 5 for a precise definition). It is known \cite{We85} that $V_r(C, \eta)$ is a Lagrangian determinantal variety of expected dimension $g-1-{r+1\choose 2}$.
{}We denote{} $\pi:\calR_g\rightarrow \calM_g$ the forgetful map. Our result is the following:
\begin{thm}\label{ramantiram}
The ramification divisor $\calQ$ of the Prym map $P:{}\calR_6{}\rightarrow \calA_5$ equals the Prym-Brill-Noether divisor in $\calR_6$, that is,
$$\calQ=\left\{(C, \eta)\in \calR_6: V_3(C, \eta)\neq 0\right\}.$$
The antiramification divisor is the pull-back of the Gieseker-Petri divisor from $\calM_6$, that is,
$\calU=\pi^*(\mathcal{GP}_{6, 4}^1)$. The divisor $\calQ$ is irreducible and reduced.
\end{thm}

As the referee pointed out to us, the irreducibility of $\mathcal{Q}$ (as well as that of $\mathcal{U}$) also follows from Donagi's results \cite{dofibers} on the monodromy of the Prym map $P:\calR_6\rightarrow \calA_5$. Apart from the Brill-Noether characterization provided by Theorem \ref{ramantiram}, the divisor $\calQ$ has yet a third (respectively a fourth!) geometric incarnation as the closure of the locus of points $(C, \eta)\in \calR_6$ with a linear series $L\in W^2_6(C)$, such that the sextic model $\varphi_L(C)\subset \PP^2$ has a totally tangent conic, see Theorem \ref{qparametrisierung} (respectively as the locus of section $(C, \eta)\in \calR_6$ of Nikulin surfaces \cite{faveni}). The rich geometry of $\calQ$ enables us to (i) compute the classes of the closures $\overline{\calQ}$ and $\overline{\calU}$ inside the Deligne-Mumford compactification $\rr_6$, then (ii) determine explicit codimension two cycles in $\calR_6$ that dominate the irreducible components of $H$. In this way we find a complete geometric characterization of $5$-dimensional ppav whose theta divisor has a non-ordinary double point. First we characterize $\theta_{\mathrm{null}}^4$ as the image under $P$ of a certain component of the intersection $\calQ \cap P^*(\theta_{\mathrm{null}})$:

\begin{thm}\label{q5}
A ppav $(A, \Theta)\in \calA_5$ belongs to $\theta_{\mathrm{null}}^4$ if an only if it is lies in the closure of the locus of Prym varieties $P(C, \eta)$, where
$(C, \eta)\in \calR_6$ is a curve with two vanishing theta characteristics $\theta_1$ and $\theta_2$, such that $$\eta=\theta_1\otimes \theta_2^{\vee}.$$ Furthermore, $\theta_{\mathrm{null}}^4$ is unirational and $[\theta_{\mathrm{null}}^4]=27\cdot 44\, \lambda_1^2 \in CH^2(\calA_5)$.
\end{thm}

Denoting by $\calQ_5\subset \calR_6$ the locus of Prym curves $(C, \eta=\theta_1\otimes \theta_2^{\vee})$ as above, we prove that $\calQ_5$ (and hence $\theta_{\mathrm{null}}^4$ which is the closure of $P(\calQ_5)$ in $\calA_5$) is unirational, by realizing its general element as a nodal curve
$$C\in \left|\mathcal{I}^2_{R_1\cdot R_2/\PP^1\times \PP^1}(5, 5)\right|,$$
where $R_1\in |\calO_{\PP^1\times \PP^1}(3, 1)|$ and $R_2\in |\calO_{\PP^1\times \PP^1}(1, 3)|$, with the vanishing theta-nulls $\theta_1$ and $\theta_2$ being induced by the projections on the two factors.

\vskip 3pt

Observing that $[H]\ne[\theta_{\mathrm{null}}^4]$ in $CH^2(\calA_5)$, the locus $H$ must have extra irreducible components corresponding to ppav with a non-ordinary singularity that occurs generically not at a two-torsion point. We denote by $H_1\subset \calA_5$ the union of these components, so that at the level of cycles $$H=\tn^{4}+H_1,$$
where $[H_1]=27\cdot 49\, \lambda_1^2$.
We have the following characterization of $H_1$:

\begin{thm}\label{thm:genus5}
The locus $H_1$ is unirational and its general point corresponds to a Prym variety $P(C, \eta)$, where
$(C, \eta)\in \calR_6$ is a Prym curve such that $\eta\in W_4(C)-W^1_4(C)$ and $K_C\otimes \eta$ is very ample.
\end{thm}

\vskip 4pt
As an application of this circle of ideas, we determine the slope of $\aa_5$.  Let $\overline{\calA}_g$ be
the perfect cone (first Voronoi) compactification of ${\calA}_g$ --- this is the toroidal compactification of $\calA_g$ constructed using the first Voronoi (perfect) fan decomposition of the cone of semi-positive definite quadratic forms with rational nullspace (see e.g.~\cite{voronoi} for the origins, and \cite{shepherdbarron} for recent progress). The Picard group of $\overline{\calA}_g$ with rational coefficients has rank $2$ (for $g\ge 2$), and it is generated by the first Chern class $\lambda_1$ of the Hodge bundle and the class of the \emph{irreducible} boundary divisor $D:=\overline{\calA}_g-\calA_g$.  The slope of an effective divisor $E\in \mathrm{Eff}(\aa_g)$ is defined as the quantity
$$s(E):=\inf\left\{\frac{a}{b}: a, b>0, \ a\lambda_1 -b[D]-[E]= c[D], \ c>0\right\}.$$
If $E$ is an effective divisor on $\aa_g$ with no component supported on the boundary and $[E]= a\lambda_1 -bD$, then $s(E):=\frac{a}{b}\geq 0$. One then defines the slope {}(of the effective cone){} of the moduli space as
$s(\aa_g):=\inf_{E\in \mathrm{Eff}(\aa_g)} s(E)$. This important invariant governs to a large extent the birational geometry of $\aa_g$; for instance $\aa_g$ is of general type if $s(\aa_g)<g+1$, and $\aa_g$ is uniruled when $s(\aa_g)>g+1$. Any effective divisor class calculation on $\aa_g$ provides an upper bound for $s(\aa_g)$.
It is known \cite{smmodfour} that $s(\aa_4)=8$ (and the minimal slope is computed by the divisor $\overline{\mathcal{J}}_4$ of Jacobians). In the next case $g=5$, the class of the closure of the Andreotti-Mayer divisor is $[\overline{N_0'}]=108\lambda_1 -14D$, giving  the upper bound $s(\aa_5)\leq \frac{54}{7}$.

\begin{thm}\label{slopea5}\footnote{Added in April 2022: The published version of the paper asserts a slightly stronger version of this result, which however does not follow from the arguments presented in this paper. This is the corrected result.}
The slope of $\aa_5$ is computed by $\overline{N_0'}$, that is, $s(\aa_5)=\frac{54}{7}$. Furthermore, the Kodaira-Iitaka dimension of $\overline{N_0'}$ is submaximal, that is, $\kappa(\aa_5, \overline{N_0'})<\mathrm{dim}(\aa_5)$.
\end{thm}

To prove this result, we define a partial compactification $\pr_6$ of $\calR_6$ and via the (rational) Prym map
$P:\pr_6\dashrightarrow {}\aa_5{}$ we investigate the pull-back
$$P^*(\overline{N_0'})= 2 \widetilde{\calQ}+\widetilde{\calU}+20\delta_0^{''},$$
where $\widetilde{\calQ}$ and $\widetilde{\calU}$ denote the closure of $\calQ$ and $\calU$ respectively in $\pr_6$, and $\delta_0^{''}$ is the divisor of degenerate Wirtinger double covers (see Section 6 for precise definitions).
Since each of the divisors appearing in this linear system admits a uniruled parametrization in terms of plane sextics having a totally tangent
conic, we are ultimately able to establish that $\overline{N_0'}$ is an extremal effective divisor on $\aa_5$.

\vskip 3pt
A final application concerns the divisor in $\aa_5$ of Pryms obtained from branched covers. The Prym variety associated to a double cover $f:\tilde{C}\rightarrow C$ branched over two points is still a ppav. When $g(C)=5$ (and only in this case), the Prym varieties constructed in this way form an irreducible divisor $\mathcal{D}^{\mathrm{ram}}:=P_*(\Delta_0^{\mathrm{ram}})$ inside the moduli space. We have the following formula for the class of the closure of $\mathcal{D}^{\mathrm{ram}}$ in $\aa_5$:
\begin{thm}\label{rampryms}
$[\overline{\mathcal{D}^{\mathrm{ram}}}]= 12 \left(153 \lambda_1-19D\right)\in CH^1(\aa_5).$
\end{thm}
Since the classes $[P^*(\overline{\mathcal{D}^{\mathrm{ram}}})]$ and $\delta_0^{\mathrm{ram}}$ are not proportional, one obtains that the general Prym variety $(A, \Theta)\in \mathcal{D}^{\mathrm{ram}}$ obtained from a ramified cover $\tilde{C}\rightarrow C$ {}(with $g(C)=5$ and $g(\tilde{C})=10$){}, is also the Prym variety induced by an \'etale cover $\tilde{C_1}\rightarrow C_1$ {}(with $g(C_1)=6$ and $g(\tilde{C_1})=11$).

\vskip 3pt

We summarize the structure of the paper. The cycle structure of $H$ and $\theta_{\mathrm{null}}^{g-1}$ is described in Section 2, whereas the classes $[\theta_{\mathrm{null}}^{g-1}], [H]\in CH^2(\calA_g)$ are computed in Section 3. The particular case $g=4$ is treated in Section 4. After some background on singularities of Prym theta divisors (Section 5), the different geometric realizations of the ramification and antiramification divisors of the Prym map $P:\calR_6\rightarrow \calA_5$, as well as the corresponding class calculations on $\pr_6$ are presented in Sections 6 and 7. A proof of Theorem \ref{slopea5}, thus determining the slope of $\aa_5$ is given in Section 8. The final sections of the paper are devoted to a complete geometric description in terms of Pryms of the two components of the cycle $H$ in genus $5$, see Theorems \ref{q5} and \ref{thm:genus5}. We close by expressing our thanks to the referee for the many pertinent comments which clearly improved the presentation of the paper.

\section{Theta divisors and their singularities}
In this section we recall notation, definitions, as well as some
results from \cite{grsmgen4}. We denote $\HH_g$ the {\it Siegel upper
half-space}, i.e.~the set of symmetric complex $g\times g$ matrices
$\tau$ with positive definite imaginary part.  If $\sigma=\begin{pmatrix} a&b\\
c&d\end{pmatrix}\in\Sp$ is a symplectic matrix in $g\times g$ block
form, then its action on $\tau\in\HH_g$ is defined by
$\sigma\cdot\tau:=(a\tau+b)(c\tau+d)^{-1}$, and the moduli space of complex principally polarized abelian variety (ppav for short) is the quotient $\calA_g=\HH_g/\Sp$, parameterizing pairs  $(A_\tau,\,\T_{\tau})$ with $A_{\tau} ={\bf C}^g/{{\bf Z}}^g\tau+{{\bf Z}}^g,$ an abelian variety and  $\T_{\tau}$ the  (symmetric) polarization bundle. We denote by $A_{\tau}[2] $ the {}group of{} two-torsion points of $A_{\tau}$. Let  $\ep,\de\in ({\bf{Z}}/2{\bf{Z}})^g$, thought of as vectors of zeros and ones; then $x=\tau\ep/2+\de /2\in A_\tau[2]$, and the shifted bundle $t_{x}^*\T$ is
still a symmetric line bundle. Up to a multiplicative constant the unique section of the above bundle is given by the {\it theta function with
characteristic $[\ep,\de]$}  defined by
$$
 \tt\ep\de(\tau,z):=\sum\limits_{m\in{{\bf Z}}^g} \exp \pi i \left[
 ^t(m+\frac{\ep}{2})\tau(m+\frac{\ep}{2})+2\ ^t(m+\frac{\ep}{2})(z+
 \frac{\de}{2})\right].
$$
We shall write $\t(\tau, z)$ for the theta function with
characteristic $[0,0]$. The zero scheme of $\t(\tau,z)$, as a function of
$z\in A_\tau$, defines the  principal polarization $\T_\tau$ on $A_\tau$.

Theta functions satisfy the heat equation
$$
 \frac{\partial^2\tt\ep\de(\tau,z)}{\partial z_j\partial z_k}
 =2\pi
 i(1+\delta_{j,k})\frac{\partial\tt\ep\de(\tau,z)}{\partial\tau_{jk}},
$$
(where $\delta_{j,k}$ is Kronecker's symbol).

The  {\it characteristic} $[\ep,\de]$ is called {\it even} or {\it odd}
corresponding to whether the scalar product ${}^t\ep\de\in{\bf{Z}}/2{\bf{Z}}$
is zero or one. Consequently, depending on the characteristic,
$\tt\ep\de(\tau,z)$ is even or odd as a function of $z$. A {\it theta
constant} is the evaluation at $z=0$ of a theta function. All odd
theta constants of course vanish identically in $\tau$.

A holomorphic function $f:\HH_g\to \bf C$ is called  a {\it  modular form} of
weight $k$ with respect to a finite index
subgroup $\Gamma\subset\Sp$ if
$$
 f(\sigma\cdot\tau)=\det(c\tau+d)^k f(\tau)\qquad\forall
 \tau\in\HH_g,\forall\sigma\in\Gamma,
$$
and if additionally $f$ is holomorphic at all cusps of
$\HH_g/\Gamma$.
Theta constants with characteristics are modular forms of
weight $\frac{1}{2}$ with respect to a finite index subgroup $\Gamma_g(4,8)\subset\Sp$.
We refer to \cite{igusabook} for a detailed study of theta functions.

We denote by
$$
 \phi:\calX_g=\HH_g\times {\bf C}^g/(\Sp\rtimes{}{\bf{Z}})^{2g})\to\calA_g=\HH_g/\Sp
$$
the universal family of ppav, and let $\T_g\subset\calX_g$ be the universal
theta divisor --- the zero locus of $\t(\tau,z)$. Following Mumford \cite{mumforddimag}, we denote by $\calS:=\Sing_{\rm
vert}\T_g$ the locus of singular points of theta divisors of ppav:
 \begin{equation}\label{sing}
\calS= \bigcup\limits_{\tau\in \calA_g}  \Sing\T_{\tau}=\left\lbrace (\tau,z)\in\HH_g\times{\bf C}^g:
  \t(\tau,z)=\frac{\p \t}{\p z_i}(\tau,z)=0,\  i=1,\dots,g\right\rbrace
\end{equation}
(computationally, by an abuse of notation, we will often work locally on $\calS$, thinking of it
as a locus inside the cover $\HH_g\times {\bf C}^g$ of $\calX_g$).
It is known that $\calS\subset\calX_g$ is of pure codimension $g+1$, and has three {}irreducible{} components \cite{civdg1}, denoted $\Sn$, $\calS_{\rm dec}$, and $\calS'$. {}Here $\Sn$ denotes{} the locus of even two-torsion
points that lie on the theta divisor, given locally by $g+1$ equations
\begin{equation}\label{Sn}
 \Sn:=\left\lbrace (\tau,z)\in\calX_g: \theta(\tau,z)=0,\ z=(\tau\ep+\de)/2\ {\rm for\ some\ } [\ep,\de]\in({{\bf Z}}/2{\bf Z})^{2g}_\even\right\rbrace.
\end{equation}

To define $\calS_{\rm dec}$, recall that a ppav is called decomposable if it is isomorphic to a product of lower-dimensional ppav. We denote then
\begin{equation}
 \calS_{\mathrm{dec}}:=\calS\cap \phi^{-1}(\calA_1\times\calA_{g-1}).
\end{equation}
Since the theta divisor of a product $(A_1,\Theta_1)\times (A_2,\Theta_2)$ is given by the union $(\Theta_1\times A_2)\cup (A_1\times\Theta_2)$, its
singular locus contains $\Theta_1\times\Theta_2$ and is of codimension 2 (see the work \cite{eila1} of Ein and Lazarsfeld for a proof of
a conjecture \cite{ardcequations} of Arbarello and De Concini that $N_{g-2}$ is in fact equal to the decomposable locus). Thus the fibers of $\calS_{\mathrm{dec}}\to\calA_{1,g-1}$ are all
of dimension $g-2$, and the codimension of $\calS_{\mathrm{dec}}\subset\calX_g$ is equal to $g+1$. (We note that any other locus of products $\calA_h\times\calA_{g-h}$ has
codimension $h(g-h)$, {}and contributes no irreducible component of $\calS${}).

Finally, $\calS'$ is the closure of the locus of singular points
of theta divisors of indecomposable ppav that are not two-torsion points.
Observe that $\calS_{\mathrm{null}}, \calS'$, and $\calS_{\mathrm{dec}}$ all come equipped with an induced structure as  determinantal subschemes of
$\mathbb H_g\times {\bf C}^g$.

The Andreotti-Mayer divisor is then defined (as a cycle) by
$$
 N_{0}:=\phi_*(\calS)=\{\tau\in\calA_g: \Sing\T_{\tau}\neq \emptyset\}.
$$
It can be shown that ${N}_{0}$ is a divisor in ${\calA_g}$,
which has at most two irreducible components, see \cite{debarredecomposes},\cite{mumforddimag}.

The {\it theta-null divisor} $\tn\subset\calA_g$ is the zero locus
of the modular form
$$F_g(\tau):=\prod_{m\, \mathrm{even}}\tt\ep\de(\tau,0).$$
Geometrically, it is the locus of ppav for which an even two-torsion point
lies on the theta divisor, and it can be shown that $\tn=\phi_*(\Sn)$, viewed as an equality of cycles. Similarly for the other
component we have $N_0'=\frac12\phi_*(\calS')$ (the one half appears because
a generic ppav in $N_0'$ has two singular points $\pm x$ on the theta divisor).

\begin{rem}
The two components of $N_0$ are zero loci of modular forms (with
some character $\chi$  in genus $1$ and $2$): $\tn$ is the zero locus of the modular form $F_{g}$
of weight $2^{g-2}(2^g +1)$, while $N_0'$ must be the zero locus of some modular
form $I_g$ of weight $g!(g+3)/4- 2^{g-3}(2^g +1)$ (the class, and thus the weight,
was computed by Mumford \cite{mumforddimag}). Unlike the explicit formula for $F_g$,
the modular form $I_g$ is only known explicitly for $g=4$, in which case it is
the so called Schottky form \cite{igusagen4},\cite{igusachristoffel}. Various approaches to constructing $I_g$ explicitly were developed in \cite{yoshikawa},\cite{krsm}.
\end{rem}

\section{Double points on theta divisors that are not ordinary double points}
We shall now concentrate on studying the local structure of a theta divisor near its singular point.
For this, we look at the tangent space to $\calS$ and the  map between the  tangent spaces.
\begin{prop}\label{pr1}
Let $x_0=(\tau_0, z_0)$  be a  smooth point  of $\calS$. Then the map
$(d\phi)_{x_0}: T_{x_0}(\calS)\rightarrow  T_{\tau_0}(N_0)$ is injective if and only if the Hessian  matrix
$$
  H(x_0):=\begin{pmatrix}
    \frac{\p^2\t}{\p z_1\p z_1}(x_0)&\ldots &\frac{\p^2\t}{\p z_1\p z_g}(x_0)\\
\vdots&\ddots&\vdots\\
  \frac{\p^2\t}{\p z_g\p z_1}(x_0)&\ldots &\frac{\p^2\t}{\p z_g\p z_g}(x_0) \end{pmatrix}
$$
 has rank $g$.
\end{prop}
\begin{proof}
Since the subvariety $\calS\subset\calX_g$ is defined by the $g+1$
equations (\ref{sing}), the point $x_0$ is  smooth if and only if the
$(\frac{g(g+1)}{2}+ g)\times (g+1)$ matrix
$$
  M(\tau_0,z_0):=\begin{pmatrix}
   \frac{\p\t}{\p\tau_{11}}&\ldots&\frac{\p\t}{\p\tau_{gg}}
   &\frac{\p\t}{\p z_1}&\ldots &\frac{\p\t}{\p z_g}\\
   \frac{\p^2\t}{\p z_1\p\tau_{11}}&\ldots&\frac{\p^2\t}{\p z_1\p\tau_{gg}}
   &\frac{\p^2\t}{\p z_1\p z_1}&\ldots &\frac{\p^2\t}{\p z_1\p z_g}\\
   \vdots&\ddots&\vdots&\vdots&\ddots&\vdots\\
   \frac{\p^2\t}{\p z_1\p\tau_{11}}&\ldots&\frac{\p^2\t}{\p z_1\p\tau_{gg}}
   &\frac{\p^2\t}{\p z_g\p z_1}&\ldots &\frac{\p^2\t}{\p z_g\p z_g} \end{pmatrix}
$$
evaluated at $x_0=(\tau_0,z_0)$ has rank  $g+1$. We compute
$$  M( \tau_0, z_0)=
  \begin{pmatrix}
   \frac{\p\t}{\p\tau_{11}}(x_0)&\ldots&\frac{\p\t}{\p\tau_{gg}}(x_0)
   &0&\ldots &0\\
   \frac{\p^2\t}{\p z_1\p\tau_{11}}(x_0)&\ldots&\frac{\p^2\t}{\p z_1\p\tau_{gg}}(x_0)
   &\frac{\p^2\t}{\p z_1\p z_1}(x_0)&\ldots &\frac{\p^2\t}{\p z_1\p z_g}(x_0)\\
   \vdots&\ddots&\vdots&\vdots&\ddots&\vdots\\
   \frac{\p^2\t}{\p z_1\p\tau_{11}}(x_0)&\ldots&\frac{\p^2\t}{\p z_1\p\tau_{gg}}(x_0)
   &\frac{\p^2\t}{\p z_g\p z_1}(x_0)&\ldots &\frac{\p^2\t}{\p z_g\p z_g}(x_0)\end{pmatrix}
$$
Since the map $ \phi$ is the projection on the  first $\frac{g(g+1)}{2}$ coordinates, the proposition follows.
\end{proof}
\begin{rem}From the heat equation for the theta function it follows that if   the Hessian  matrix  $ H(x_0)$
has rank $g$, then $x_0$ is a  smooth point  of $\calS$.

We also note that from the product rule for differentiation and the heat equation it follows that the second derivative
$$
 \frac{\partial^2\tt\ep\de(\tau,z)}{\partial z_j\partial z_k}|_{z=0}
$$
restricted to the locus $\tt\ep\de(\tau,0)=0$ is also a modular form for $\Gamma_g(4,8)$.
\end{rem}

Since we have different, easier to handle, local defining equations (\ref{Sn}) for $\Sn$, we can obtain better results in this case.
\begin{prop}
A point $x_0\in\Sn$ is a  smooth point of $\Sn$ unless  $\frac{\p\t}{\p\tau_{ij}}(x_0)=0$ for all $ 1\leq i, j\leq g$.
The map $(d\phi)_{x_0}$ is injective if and only if the Hessian  matrix  $ H(x_0)$   has rank $g$.
\end{prop}
\begin{rem}
If $x_0=(\tau_0,z_0)$ is a  smooth point  of $\Sn$, while $\tau_0$ is singular in $\tn$,
this implies that at least two different theta constants vanish at $\tau_0$.
\end{rem}
Using the above framework, we get a complete description of the intersection $\Sn\cap \calS'$, obtaining thus an easier proof of one of the main results of \cite{grsmordertwo}.
\begin{prop}
For $x_0\in\Sn$, the point $x_0$ lies in $\calS'$ if and only if the rank of $H(x_0)$ is less
than $g$.
\end{prop}
\begin{proof}
If $x_0\in \calS'{}\cap\Sn{}$, then it is a singular point in $\calS$, hence the rank of $H(x_0)$ is less than $g$
by the above proposition. To obtain a proof in the other direction, since $z_0$ is a two-torsion point, the matrix  $  M( \tau_0, z_0)$
appearing in the proof of the proposition above has the form
$$
 M(x_0)=\begin{pmatrix}
   \frac{\p\t}{\p\tau_{11}}(x_0)&\ldots&\frac{\p\t}{\p\tau_{gg}}(x_0)
   &0&\ldots &0\\
 0&\ldots&0
   &\frac{\p^2\t}{\p z_1\p z_1}(x_0)&\ldots &\frac{\p^2\t}{\p z_1\p z_g}(x_0)\\
   \vdots&\ddots&\vdots&\vdots&\ddots&\vdots\\
0&\ldots&0
   &\frac{\p^2\t}{\p z_g\p z_1}(x_0)&\ldots &\frac{\p^2\t}{\p z_g\p z_g}(x_0)
   \end{pmatrix}
$$
Hence if  the rank of $H(x_0)$ is less than $g$, $x_0$ is a singular point of $\calS$; thus either it is a singular point of $\Sn$, or it lies  in the intersection $\Sn\cap\calS'$.
The first case cannot happen for dimensional reasons (the singular locus of $\Sn$ is codimension at least 2 within $\Sn$, see also \cite{civdg1}), and thus we must have $x_0\in\Sn\cap\calS'$.
\end{proof}

\begin{cor}
Set theoretically we have
 $$\phi(\Sn\cap\calS')=\tn^{g-1}.$$
\end{cor}

\begin{rem}
From the previous proof it also follows that $\Sing\Sn\subset \Sn\cap \calS'$.
\end{rem}

Our further investigation  will  consider the subvariety
$$
 \calH:=\calS'^{g-1}:=\{x_0=(\tau_0, z_0)\in \calS': \rk H(x_0)<g\}\subset\calX_g
$$
(notice that since the derivative of a section of a
line bundle is a section of the same bundle when restricted to the zero locus of the section,
this is an algebraic subvariety of $\calX_g$). Note that $\calH$, being defined by explicit equations in the (derivatives of) theta functions, comes equipped with a scheme structure. Then we define the pushforward cycle
$$
 2H:=2N_0'^{g-1}:=\phi_*(\calH)\subset\calA_g
$$

Unlike in the case of the theta-null, $\calH\not\subset\Sing\calS$. Indeed, if
$z_0$ is not a two-torsion point, the condition $\rk H(x_0)<g$
does not imply that $\rk M(\tau_0, z_0)<g+1$, as the matrix $M$ at $z_0$ does not have as many zero entries as
in the theta-null case. Still, we have the set-theoretic inclusions
$$
 \Sn\cap \calS'\subset\calH\quad\mbox{and}\quad \tn^{g-1}\subset H.
$$

The locus $\calH$ is given locally by
$(g+2)$ equations (the $g+1$ equations for $\calS'$ together with the vanishing of the Hessian determinant), and thus
each irreducible component of $\calH$ has codimension at most $g+2$ in $\calX_g$. However, we note that $\calS_{\rm dec}\subset\calH\subset\calS$ is an irreducible component of codimension $g+1$. We now check that all other irreducible components of $\calH$ are indeed of expected codimension $g+2$.
Indeed, we first note that by the results of Ciliberto and van der Geer \cite{civdg2} the Andreotti-Mayer locus $N_k$ (parameterizing ppav whose  theta divisor has  singular locus of dimension at least
$k$) with $1\leq k\leq g-3$ has codimension at least $k+2$ in $\calA_g$, and thus its preimage in $\calS$ cannot be an irreducible component of $\calS$ for dimension reasons. Now for both $\calS'$ and $\calS_{\rm null}$ it is known that generically the singular points of the theta divisors are ordinary double points, and thus $\calH$ cannot be equal to either of these loci. Finally, by the results of Ein and Lazarsfeld \cite{eila1} the locus $N_{g-2}$ is equal to the locus of indecomposable ppav, and each component $\calA_h\times\calA_{g-h}$ of it has codimension too high, except for $h=1$.

The above discussion leads to the following result:
\begin{prop}
The Andreotti-Mayer locus $N_1$ is contained  in  $H$.
\end{prop}
\begin{proof}
Indeed, for $\tau_0\in N_1$ we let $z(t)\subset\Sing\T_\tau$ be a curve of singular points such that $z(0)=z_0$ is a smooth
point of the curve.
Differentiating (\ref{sing}) with respect to $t$, we get $g$ non-zero equations (the derivative of the first one will vanish):
$$
 \sum_{j=1}^g
  \frac{\p^2\theta(\tau_0, z(t))}{\p z_i\p z_j}\frac{\p z_j(t)}{\p t}=0.
$$
Denoting
$$
 v:=\left(\frac{\p z_1}{\p t}, \dots, \frac{\p z_g}{\p t} \right)|_{t=0}
$$
this means that $H(x_0)\cdot v=0$, and since by our assumption $z_0$ is a smooth point of the curve and thus
$v\ne 0$, the matrix $H(x_0)$ has a kernel, and in particular is not of maximal rank.
\end{proof}

\begin{cor}\label{Mgin}
For $g\geq 5$ the locus of Jacobians $\calJ_g$ is contained in  $H$. Hence for $g\geq 5$ set-theoretically $\tn^{g-1}\subsetneq  H$.
\end{cor}
\begin{proof}
Indeed, we have $\calJ_g\subset N_1\subset  H$ for $g\ge 5$. However, since
for all $g$ the divisor $\tn$ does not contain $\calJ_g$, we must have $\calJ_g\subset \overline{H\setminus\tn^{g-1}}$.
\end{proof}

\section{Class computations in cohomology}
In this section we compute the class of the components of the expected dimension of the loci $\calH$ and  $H$ in Chow and cohomology rings (our computation works in both, as we only use Chern
classes of vector bundles) of $\calX_g$ and $\calA_g$, respectively.

Recall that Mumford \cite{mumforddimag} computed the class of $N'_0$ in the Picard  group of the partial toroidal compactification of the moduli space $\calA_g$ (the class of $\tn$ is easier, and was computed previously {}by Freitag \cite{freitagbooksiegel}{}). We shall compute the classes  of the codimension 2  cycles  $H$ and $\tn^{g-1}$ on $\calA_g$. As  a consequence we will obtain  a complete description of  $H$ in genus 4, rederive some result of \cite{grsmgen4}, and reprove that for $g\geq 5$ the locus $H$ has other components besides $\tn^{g-1}$. Debarre \cite[Section 4]{debarredecomposes} computed the class of the intersection $\tn\cap N_0'$ and used this to show that this intersection is not irreducible. In spirit our computation is similar, though much more involved.

For the universal family $\phi:\calX_g\to\calA_g$ we denote by $\Omega_{\calX_g /\calA_g}$ the relative cotangent bundle,
by $\EE:=\phi_*\Omega_{\calX_g/\calA_g}$ we denote its pushforward --- the rank $g$ vector bundle that is called the
Hodge bundle. Then the Hodge class $\lambda_1:=c_1(\EE)$ is the
Chern class of the line bundle of modular forms of weight one on $\calA_g$.

The basic tool for our computation of pushforwards is the following:
\begin{lm}
The pushforward under $\phi$ of powers of the universal theta divisor $\T\subset\calX_g$ can be computed as follows:
$$
 \phi_*([\T^k])=\begin{cases}
 0&{\rm if\ }k<g\\
 g!&{\rm if\ }k=g\\
 \frac{(g+1)!}{2} \lambda_1 &{\rm if\ }k=g+1\\
 \frac{(g+2)!}{8} \lambda_1 ^2&{\rm if\ }k=g+2
\end{cases}
$$
\end{lm}
\begin{proof}
The first three cases are consequence of the computation in \cite[page 373]{mumforddimag}. The last  case is the next step of  the same computation, recalling that $c_2(\EE)=\lambda_1 ^2/2$. In full generality the pushforwards of the universal theta divisor were computed and studied in \cite{vdgeercycles} (note that the universal theta divisor trivialized along the zero section, that is, the class $[\Theta]-\lambda_1/2$, is used there, and it is shown that $\phi_*\left(([\Theta]-\lambda_1/2)^k\right)=0$ unless $k=g$).
\end{proof}

Note that the locus $\calS$ is given as the scheme of zeroes of theta function and its derivatives, i.e.~given by zeroes of a section of $\Omega_{\calX_g /\calA_g}(\Theta)\otimes\calO_{\calX_g}(\T)$ (see \cite{mumforddimag}). Hence
$$[N_0]=\phi_*\left(c_g\left(\Omega_{\calX_g /\calA_g}(\Theta)\otimes_{\calO_{\calX_g}}\calO_{\T}\right)\right).$$
Recall now that $\calS^{g-1}\subset\calS$ is defined by the equation ${\rm det}\ H(x_0)=0$. On $\calS$, each second
derivative of the theta function is a section of $\Theta$, and the determinant of the
Hessian matrix is known (see \cite{grsmordertwo},\cite{rdejong})  to be a section of
$$
  \calO_{\calX_g}(g\T)\otimes \phi^*(\det\EE)^{\otimes 2}\otimes\calO_\calS.
$$
Using the above formula for the class of $\calS$, to get $H$ we will need to compute the pushforward
$$
  \phi_*\left(c_g\left(\Omega_{\calX_g /\calA_g}(\Theta)_{|\T}\right)\cdot( g\T+2\phi^*\lambda_1) \right)
$$

The computation  becomes rather delicate since $\calS^{g-1}$ is not equidimensional. We set
$$
  \calS_{\mathrm{indec}}:=\overline{\calS\setminus\calS_{\mathrm{dec}}}=\calS'\cup\calS_{\rm null},
$$
which is then purely of codimension $g+2$ in $\calX_g$, and thus we have
$$
 [\calS_{\mathrm{indec}}^{g-1}]=[\calS_{\mathrm{indec}}] \cdot( g[\T]+2\lambda_1)\in CH^{g+2}(\mathcal{X}_g).
$$

However, for dimension reasons it turns out that we often do not need to
deal with the class of $\calS_{\mathrm{dec}}$:
\begin{prop}\label{propunu}
For $g\geq 4$ we have the  equality of codimension 2 classes on $\calA_g$:
$$
 [N_0^{g-1}]=[N_{0,\, \mathrm{indec} }^{g-1}]:=[\phi_*(\calS_{\mathrm{indec}}^{g-1})]
$$
Moreover this class can be computed as
$$
 [N_0^{g-1}]=\frac{g!}{8}(g^3+ 7g^2+18g+24)\lambda_1^2\in CH^2(\calA_g).
$$
\end{prop}
\begin{proof}
The first statement is a consequence of the fact that  the map $\phi$ has $(g-2)$-dimensional fiber along
$\calS_{\mathrm{dec}}$, and generically $0$-dimensional fibers over $\calS_{\mathrm{indec}}$. (Note that this is the place in the argument
where we are using the assumption $g\ge 4$ to ensure that $\calS_{\mathrm{indec}}^{g-1}$ is in fact non-empty, and that the codimension
of its image under $\phi$ is lower than the codimension of $\calA_1\times\calA_{g-1}$.{})
We now compute
$$
\begin{aligned}
 &[N_0^{g-1}]=\phi_*\left(c_g\left(\Omega_{\calX_g /\calA_g}(\Theta)\cdot\T\right)\cdot( g\T+2\phi^* \lambda_1 )\right)\\
 &=\phi_*\left((\T^g+\T^{g-1}\phi^*\lambda_1 +\T^{g-2}\phi^*\lambda_2+\dots)\cdot  (g\T^2+2\T \phi^* \lambda_1 )\right)\\
 &=\phi_*\left(g\left(\T^{g+2}+\T^{g+1}\phi^* \lambda_1 + \T^{g}\frac{\phi^* \lambda_1 ^2}{2}\right)+(2\T^{g+1}\phi^* \lambda_1 +2 \T^{g}\phi^* \lambda_1 ^2)\right)\\
 &=\left(\frac{g(g+2)!}{8}+\frac{g(g+1)!}{2}+\frac{g(g)!}{2}+ (g+1)!+ 2g!\right)\lambda_1 ^2\\
 &=\frac{g!}{8}(g^3+ 7g^2+18g+24)\lambda_1 ^2.
\end{aligned}
$$
\end{proof}

We now compute the class of the locus $\tn^{g-1}$: recall that a theta constant is
a modular form of weight $\frac{1}{2}$, and that the determinant of the Hessian matrix of $\tt\ep\de(\tau,z)$ evaluated
at $z=0$ is a modular form of weight $\frac{g+4}{2}$ along the zero locus of
$\tt\ep\de(\tau)$ (see \cite{grsmgen4},\cite{rdejong}). We thus get:
\begin{prop}\label{propdoi}
For   $g\geq 2$ we have
$$
 [\tn^{g-1}]=(g+4)2^{g-3}( 2^g+1)\lambda_1^2.
$$
\end{prop}
\begin{proof}

Indeed, we have
$$
 \tn^{g-1}=\left\lbrace\tau\in\HH_g: \exists[\ep,\de]\ {\rm even}\ ,
 \tt\ep\de(\tau)=\det\left(
 \frac{\partial^2\tt\ep\de(\tau,z)}{\partial z_j\partial
 z_k}\right)|_{z=0}=0\right\rbrace,
$$
Since there are $2^{g-1}( 2^g+1)$ even characteristics, for each of them we get a contribution of $\lambda_1/2$
(for the zero locus of the corresponding theta constant) times $(g+4)\lambda_1/2$ (for the Hessian).
\end{proof}

The proof of Theorem \ref{thm:class} comes by subtraction using the class formulas established in Propositions
\ref{propunu} and \ref{propdoi}, while taking into account the relation given in formula (\ref{cup}).

\section{The case $g=4$}
In this section we will work out the situation for genus 4 in detail, eventually proving Theorem \ref{thm:genus4}. By the above formulas for $g=4$ we have
$$
 [\tn^{3}]=272\lambda_1 ^2;\qquad
 [N_0^{3}]=3\cdot 272\lambda_1 ^2.
$$
Moreover, going back from $N_0^{g-1}$ to  $H=N_0'^{g-1}$, we recall that for arbitrary genus by definition we have $N_0^{g-1}=\tn^{g-1}+2 H$, and since
at the intersection of the two components $\tn$ and $N_0'$ the singular points lie on both, we also have that set-theoretically
$$
 \tn^{g-1}\subset H
$$
As an immediate consequence we obtain:
\begin{prop}
The following identity holds at the level of codimension two cycles on $\calA_4$:
$$N_0^{3}=3\tn^{3}.$$
\end{prop}
\begin{proof}From the formulas above we see that the cycle $2\tn^{3}$ appears inside $2N_0'^{3}$, and thus that $3\tn^{3}$ is a subcycle of $N_0^3$.
Since the Chern classes are equal and $\tn^3$ is equidimensional, we need to rule out the possibility of $N_0^{3}$
having an extra lower dimensional component. However, for genus 4 we know geometrically that $N_0'$ is
the locus of Jacobians. Using Riemann's Singularity Theorem for genus 4 curves we then see
that the period matrix of a Jacobian is in $N_0'^{3}$ if and only if its theta divisor is singular at a two-torsion point,
i.e.~if this Jacobian lies in $\tn^3$ (notice that this reproves a result of \cite{grsmgen4}).
\end{proof}

The proof of Theorem \ref{thm:genus4} is an immediate consequence of the above facts.
We can prove something more: let $I_4$ be the Schottky modular form of weight 8 defining the Jacobian locus. Let then
$$
{\rm det}\, \calD (I_4):= {\rm det}\left(\begin{array}{rrrr}
 \,\frac{\p I_4}{\p\tau_{11}}&\frac{1}{2}\frac{\p I_4}
 {\p\tau_{12}}&\dots&\frac{1}{2}\frac{\p I_4}{\p\tau_{1 4}}\\
 \frac{1}{2}\frac{\p I_4}{\p \tau_{21}}&\frac{\p I_4}{\p
 \tau_{22}}&\dots&\frac{1}{2}\frac{\p I_4}{\p\tau_{2 4}}\\
 \dots&\dots&\dots&\dots\\
 \frac{1}{2}\frac{\p I_4}{\p \tau_{4 1}}& \dots&\dots& \,\
 \frac{\p I_4}{\p\tau_{4 4}}\end{array}\right)
$$
The restriction of this determinant to the zero locus of $I_4$ is a modular form of weight $34=8\cdot 4+2$.
By Proposition \ref{pr1} we know that for a point in $N_0'\setminus  H$ the matrix $\calD(I_4)$ is proportional to the Hessian matrix $H(x_0)$, hence it vanishes exactly along $\tn^3$. The class of the cycle
$$ \lbrace I_4=\det\calD(I_4)=0\rbrace$$
is $8\cdot34\, \lambda_1 ^2=272\, \lambda_1 ^2$. Thus we obtain
\begin{prop}
The locus $\tn^3\subset\calA_4$ is a complete intersection given by
$$
 I_4=\det\calD(I_4)=0.
$$
\end{prop}
{}
We observe that, by Riemann's Theta Singularity theorem,  this is the locus of Jacobians  with  theta divisor  singular at a two-torsion point.
Moreover, the form $\sqrt {F_4}$ (recall that $F_g$ is the product of all even theta constants) is well-defined along  the  Jacobian  locus and it has the same weight, hence we get a different proof of the following result recently  obtained by Matone and Volpato \cite{mavo}:
\begin{cor}
On  $\calJ_4$ we have the equality $\sqrt {F_4}=c\det\calD(I_4)$ for some constant $c$.
The locus $\tn^3$ can also be given by equations
$$I_4=\sqrt {F_4}=0.$$
\end{cor}
{}
Contrary to the situation in genus 4, for higher genera we know that we have other components, see Corollary \ref{Mgin}.  This fact can also be deduced from our class computation as follows.
\begin{proof}[Proof of Proposition \ref{prop:ne}]
Recall that the statement we are proving is that at the level of effective cycles $\tn^{g-1}\subsetneq  H$ for any $g\ge 5$. We first note that the above discussion for the genus 4 case shows that the cycle-theoretic inclusion holds. Secondly, since we have computed both classes, we see that for $g\ge 5$ the class of $N_0^{g-1}$
is not equal to 3 times the class of $\tn^{g-1}$. In fact the growth orders of the degrees of these two classes are respectively
$$\deg\tn^{g-1}\sim 4^{g-4}\deg\tn^3;\qquad
\deg N_0^{g-1}\sim\frac{g!}{4!}\deg N_0^3,$$
and one would thus expect many additional components.
\end{proof}
The rest of the paper is devoted to studying the geometry for $g=5$ in detail; in this case we will be able to describe all components explicitly, and will also obtain many results describing the classical Prym geometry of the situation.

\section{Prym theta divisors and their singularities }
While for higher $g$ the geometry of the locus $H\subset \calA_g$ appears quite intricate, for $g=5$ one can use the Prym map $P: \calR_6\rightarrow \calA_5$. We begin by setting the notation and reviewing the basic facts about Prym varieties and their moduli, which will be used throughout the rest of the paper.

Let $\calR_g$ be the moduli space of pairs $(C,\eta)$ with $[C]\in\calM_g$, and $\eta$ a non-zero two-torsion point of the Jacobian $\Pic^0(C)$. We denote by $f: \tilde C \to C$ the \'etale double cover induced by $\eta$ (so the genus of $\tilde C$ is equal to $2g-1$), by
$i: \tilde C \to \tilde C$ the involution exchanging the sheets of $f$, and by $\varphi_{K_C\otimes
\eta}: C\rightarrow \PP H^0(C, K_C\otimes \eta)^{\vee}$ the Prym-canonical map. The map $\varphi_{K_C\otimes \eta}$ is an embedding if and only if $\eta\notin C_2-C_2$ (where we denote $C_k:=\Sym^k(C)$).

We recall the definition of the Prym map $P:\calR_g\rightarrow \calA_{g-1}$. Consider the norm map
$\mathrm{Nm}_f: \Pic^{2g-2}(\tilde C) \to \Pic^{2g-2}(C)$ induced by the double cover $f$. The even component of the preimage
$$
  \mathrm{Nm}_f^{-1}(K_C)^+ :=\left\lbrace L \in \Pic^{2g-2}(\tilde C): \mathrm{Nm}_f(L)=K_C, \ h^0(\tilde{C}, L) \equiv 0 \mod 2 \right\rbrace
$$
is then an abelian variety of dimension $g-1$.
Denoting by $\Theta_{\tilde C}\subset \Pic^{2g-2}(\tilde{C})$ the Riemann theta divisor, scheme-theoretically we have the following equality $\Theta_{\tilde C}|_{\mathrm{Nm}_f^{-1}(K_C)^+} = 2\Xi$, where $\Xi$ is a principal polarization.  The {\em Prym variety} is defined to be the ppav
$$
 P(C,\eta):=\left(\mathrm{Nm}_f^{-1}(K_C)^+, \Xi\right)\in \calA_{g-1}.
$$
The polarization divisor can be described explicitly following \cite{Mumfordprym}:
$$
  \Xi(C, \eta) :=\{L\in \mathrm{Nm}_f^{-1}(K_C)^+: h^0(\tilde{C}, L)>0\}.
$$
A key role in what follows is played by the \emph{Prym-Petri} map
$$
  \mu^-_L: \wedge^2 H^0(\tilde C, L) \to H^0(C, K_C\otimes \eta),\quad u\wedge v\mapsto u\cdot i^*(v)-v\cdot i^*(u),
$$
where one makes the usual identification $H^0(C, K_C\otimes \eta)=H^0(\tilde{C}, K_{\tilde{C}})^-$ with the $(-1)$ eigenspace under the involution $i$. Following \cite{We85}, for $(C, \eta)\in \calR_g$ and $r\geq -1$,  we define the determinantal locus
$$
  V_r(C, \eta):=\{L\in \mathrm{Nm}_f^{-1}(K_C): h^0(L)\geq r+1, \ h^0(L)\equiv r+1 \mod 2\}.
$$
For a general Prym curve $(C, \eta)\in \calR_g$,  the map $\mu^-_L$ is injective for every $L\in \mathrm{Nm}_f^{-1}(K_C)$, and $\dim V_r(C, \eta)=g-1-{r+1\choose 2}$, see \cite{We85}.

For a point $L \in \Xi$, we recall the description of the tangent cone $TC_{L}(\Xi)$. Suppose $h^0(\tilde C, L) = 2m\geq 2$ and we fix a basis
$\{s_1 \dots, s_{2m}\}$ of $H^0(\tilde C, L)$. Consider the skew-symmetric matrix
$$
  M_{L} := \left(\mu^-_L(s_k\wedge s_j)\right)_{1\leq k , j \leq 2m}
$$
and the pfaffian
$\op{Pf}(L) := \sqrt {\det(M_{L})}\in \Sym^m H^0(C, K_C\otimes \eta)$. Via the identification $T_L(P(C, \eta))=H^0(C, K_C\otimes \eta)^{\vee}$ we have the following result of \cite{Mumfordprym}:
\begin{thm} If $h^0(\tilde C, L) =2$ then $\mathrm{Pf}(L) = 0$ is the equation  of the projectivized tangent space $\PP T_L(\Xi)$. If  $m \geq 2$ then $L \in \Sing(\Xi)$ and either $\mathrm{Pf}(L)\equiv 0$, in which case $\mathrm{mult}_L(\Xi)\geq m+1$, or else, $\mathrm{Pf}(L)=0$ is the equation of the tangent cone $\PP TC_L(\Xi)$.
\end{thm}
Note that one can have $L\in \Sing(\Xi)$ even when $m = 1$ and $\op{Pf}(L)$ is identically zero, so that the Prym theta divisor $\Xi$ can have two types of singularities, as follows:
\begin{df} For a point $L \in \Sing(\Xi)$, one says that \par \noindent
(1) $L$ is a \emph{stable} singularity if $h^0(\tilde C, L) = 2m \geq 4$, \par \noindent
(2) $L$ is an \emph{exceptional} singularity if $L=f^*(M)\otimes \calO_{\tilde{C}}(B)$, where $M\in \mathrm{Pic}(C)$ is a line bundle with $h^0(C, M) \geq 2$ and $B$ is an effective divisor on $C$.
\end{df} \noindent
Let $\Sing^{\mathrm{st}}_f(\Xi)=V_3(C, \eta)$ be the locus of stable singularities and $\Sing^{\mathrm{ex}}_f(\Xi)$  the locus{} of exceptional singularities.  Clearly
$\Sing(\Xi)=\Sing^{\mathrm{st}}_f(\Xi)\cup \Sing^{\mathrm{ex}}_f(\Xi)$. Both these notions depend on the \'etale double cover $f:\tilde{C}\rightarrow C$  and are not intrinsic to $\Xi$. Furthermore, there can be singularities that are simultaneously stable and exceptional. Every singularity of a $4$-dimensional theta divisor $\Xi$ can in fact be realized as both a stable and an exceptional singularity in different incarnations of $(A, \Xi)\in \calA_5$ as a Prym variety.

For a decomposable vector $0\neq u \wedge v \in \wedge^2 H^0(\tilde{C}, L)$, we set
$$
\op{div}(u) := D_u + B,\qquad\op{div}(v) := D_v + B,
$$
where $D_u, D_v$ have no common components and $B\geq 0$ is an effective divisor on $C$.
The next lemma is well known, see \cite[Appendix C]{ACGH}:
\begin{lm}\label{excsing1} For $0\neq u \wedge v\in \wedge^2 H^0(\tilde C, L)$ the  following are equivalent.
\begin{enumerate}
\item $\mu_L^-(u \wedge v)=0.$
\item $D_u, D_v \in |f^*M|$ where $M \in \Pic(C)$ with $h^0(C, M)\geq 2$.
\end{enumerate}
\end{lm}

\noindent In such a case we write $L=f^*(M)\otimes \calO_{\tilde{C}}(B)$, hence $K_C=M^{\otimes 2}\otimes \calO_C(f_*(B))$, in particular $h^0(C, K_C \otimes M^{\otimes (-2)}) \geq 1$, and the Petri map $\mu_0(M)$ is not injective.  In particular, $\Sing_f^{\mathrm{ex}}(\Xi)=\emptyset$ if $C$ satisfied the Petri theorem.

\vskip 3pt
Suppose $L\in V_3(C, \eta)$ is a \emph{quadratic stable singularity}, hence $h^0(\tilde C, L) = 4$ and
$\mathrm{Pf}(L) \neq 0$. Setting $\PP^5 := \PP(\wedge^2 H^0(L)^{\vee})$ and $\PP^{g-2} := \PP(H^0(K_C \otimes \eta)^{\vee})$, we consider the
projectivized dual of the Prym-Petri map
$$
\delta:=\PP\left((\mu^{-}_L)^{\vee}\right): \PP^{g-2} \to \PP^5.
$$
The Pl\"ucker embedding of the Grassmannian $\GG^* := G(2, H^0(L)^{\vee})\subset \PP^5$  is a rank $6$ quadric whose preimage $Q_L:=\delta^{-1}(\GG^*)$ is defined precisely by the Pfaffian $\mathrm{Pf}(L)$.  Note also that
$\rk(Q_{L}) \leq \rk(\mu^-_L).$
On the other hand let $\GG := G(2, H^0(L)) \subset \PP(\wedge^2 H^0(L))$ be the dual Grassmannian. It is again
a standard exercise in linear algebra to show the equivalence
$$
\rk(Q_{L}) \leq 4 \Longleftrightarrow \GG \cap \PP\left(\Ker(\mu^-_L)\right) \neq \emptyset.
$$
For a point $L\in \Sing^{\mathrm{st}}_f(\Xi)$ one has $\op{mult}_L(\Xi)=2$
 if and only if  $h^0(C, M) \leq 2$ for any line bundle $M$ on $C$ such that
$h^0(\tilde C, L \otimes f^* M^{\vee}) \geq 1$, see \cite{SV04}. We summarize this discussion as follows:

\begin{prop}\label{rk4sing} For a quadratic singularity $L\in \mathrm{Sing}_f^{\mathrm{st}}(\Xi)$ the following conditions are equivalent:
\begin{enumerate}
\item $\rk(Q_{L})\leq 4$.
\item $\GG\cap\PP\left(\Ker(\mu^-_L)\right) \neq \emptyset$.
\item $L \in \Sing^{\mathrm{st}}_{f}(\Xi) \cap \Sing^{\mathrm{ex}}_{f}(\Xi)$.
\end{enumerate}
\end{prop}

\section{Petri divisors and the Prym map in genus $6$}
This section is devoted to the study of singularities of Prym theta divisors of dimension $4$ via
the Prym map
$P:\calR_6 \to \calA_5$.

We review a few facts about the Deligne-Mumford compactification $\rr_g$ of $\calR_g$, and refer to \cite{dofibers} and \cite{falu} for details. The space $\rr_g$ is the coarse moduli space associated to the Deligne-Mumford stack $\overline{\bf{R}}_g$ of stable Prym curves of genus $g$. The geometric points of $\rr_g$ correspond to triples $(X, \eta, \beta)$, where $X$ is a quasi-stable curve with $p_a(X)=g$, $\eta\in \Pic(X)$ is a line bundle of total degree $0$ on $X$ such that $\eta_{E}=\calO_E(1)$ for each smooth rational component $E\subset X$ with $|E\cap \overline{X-E}|=2$ (such a component is called exceptional), and $\beta:\eta^{\otimes 2}\rightarrow \calO_X$ is a sheaf homomorphism whose restriction to any non-exceptional component is an isomorphism. Denoting $\pi:\rr_g\rightarrow \mm_g$ the forgetful map, one has the formula \cite[Example 1.4]{falu}
\begin{equation}\label{pullbackrg}
\pi^*(\delta_0)=\delta_0^{'}+\delta_0^{''}+2\delta_{0}^{\mathrm{ram}}\in CH^1(\rr_g),
\end{equation}
where $\delta_0^{'}:=[\Delta_0^{'}], \, \delta_0^{''}:=[\Delta_0^{''}]$, and $\delta_0^{\mathrm{ram}}:=[\Delta_0^{\mathrm{ram}}]$ are boundary divisor classes on $\rr_g$ whose meaning we recall. Let us fix a general point $[C_{xy}]\in \Delta_0$ corresponding to a smooth $2$-pointed curve $(C, x, y)$ of genus $g-1$ and the normalization map $\nu:C\rightarrow C_{xy}$, where $\nu(x)=\nu(y)$. A general point of $\Delta_0^{'}$ (respectively of $\Delta_0^{''}$) corresponds to a stable Prym curve $[C_{xy}, \eta]$, where $\eta\in \Pic^0(C_{xy})[2]$ and $\nu^*(\eta)\in \Pic^0(C)$ is non-trivial
(respectively, $\nu^*(\eta)=\calO_C$). A general point of $\Delta_{0}^{\mathrm{ram}}$ is of the form $(X, \eta)$, where $X:=C\cup_{\{x, y\}} \PP^1$ is a quasi-stable curve with $p_a(X)=g$, whereas $\eta\in \Pic^0(X)$ is a line bundle characterized by $\eta_{\PP^1}=\calO_{\PP^1}(1)$ and $\eta_C^{\otimes 2}=\calO_C(-x-y)$.

For $1\leq i\leq [\frac{g}{2}]$ we have a splitting of the pull-back of the boundary
\begin{equation}
\pi^*(\delta_i)=\delta_i+\delta_{g-i}+\delta_{i:g-i}\in CH^1(\rr_g),
\end{equation}
where the boundary classes $\delta_i:=[\Delta_i], \delta_{g-i}:=[\Delta_{g-i}]$ and $\delta_{i:g-i}:=[\Delta_{i:g-i}]$ correspond to the possibilities of choosing a pair of two-torsion line bundles on a smooth curve of genus $i$ and
one of genus $g-i$, such that the first one, the second one, or neither of the corresponding bundles is trivial, respectively, see \cite{falu}.

Often we content ourselves with working on the partial compactification $\widetilde{\calR}_g:=\pi^{-1}(\calM_g\cup \Delta_0)$ of $\calR_g$.
When there is no danger of confusion, we still denote by $\delta_0^{''}, \delta_0^{''}$ and $\delta_0^{\mathrm{ram}}$ the restrictions of the corresponding boundary classes to $\pr_g$. Note that $CH^1(\pr_g)={\bf Q}\langle \lambda, \delta_0^{'}, \delta_0^{''}, \delta_0^{\mathrm{ram}}\rangle$.

\vskip 3pt
The extension of the (rational) Prym map $P:\rr_g\dashrightarrow \overline{\mathcal{A}}_{g-1}$  over the general point of each of the boundary divisors of $\rr_g$ is well-understood, see e.g.~\cite{dofibers}. The Prym map contracts $\Delta_0^{''}$ and all boundary divisors $\pi^*(\Delta_i)$ for $1\leq i\leq \lfloor\frac{g}{2}\rfloor$. The Prym variety corresponding to a general point $[C_{xy}, \eta]\in \Delta_0^{''}$ as above is the Jacobian $\Jac(C)$ of the normalization. Thus $P(\Delta_0^{''})=\mathcal{J}_{g-1}$. The pullback map $P^*$ on divisors has recently been described in \cite{grsmprym}: one has
\begin{equation}\label{prympull}
P^*(\lambda_1)=\lambda-\frac{\delta_0^{\mathrm{ram}}}{4},\quad P^*(D)=\delta_0^{'}.
\end{equation}
\begin{rem}
We sketch an alternative way of deriving the first formula in (\ref{prympull}). For each $(C, \eta)\in \calR_g$, there is a canonical identification of vector bundles $T_{P(C, \eta)}^{\vee}=H^0(C, K_C\otimes \eta)\otimes \calO_{P(C, \eta)}$. The pull-back $P^*(\mathbb E)$ of the Hodge bundle can be identified with the vector bundle $\calN_1$ on  $\rr_g$ with fiber $\calN_1(C, \eta)=H^0(C, \omega_C\otimes \eta)$, over each point
$(C, \eta)\in \rr_g$ (we skip details showing that this description carries over the boundary as well). Therefore $P^*(\lambda_1)=c_1(\calN_1)=\lambda-\frac{1}{4}\delta_0^{\mathrm{ram}}$, where we refer to \cite{falu} for the last formula.
\end{rem}

We have seen that for $[\tilde{C}\stackrel{f}\rightarrow C]\in \calR_g$  with $\Sing^{\mathrm{ex}}_f(\Xi)\neq \emptyset$,
the curve $C$ fails the Petri theorem. Let $\mathcal{GP}_{g, k}^1  \subset \mathcal M_g$
denote the \emph{Gieseker-Petri} locus whose general element is a curve $C$ carrying a globally generated pencil
$M \in W^1_k(C)$ with $h^0(C, M) = 2$, such that the multiplication map
$$\mu_0(M): H^0(C, M) \otimes H^0(C, K_C \otimes M^{\vee}) \to H^0(C, K_C)$$ is not injective.
It is proved in \cite{fapetri} that for $\frac{g+2}{2}\leq k\leq g-1$, the locus $\mathcal{GP}_{g, k}^1$ has a divisorial component.
As usual, we denote by $\calM_{g, d}^r$ the locus of curves $[C]\in \calM_g$ such that $W^r_d(C)\neq \emptyset$.

In the case of $\calM_6$ there are two Gieseker-Petri loci, both irreducible of pure codimension $1$, {}described as follows:{}

\noindent
$\bullet$ The locus $\mathcal{GP}_{6,4}^1$ of curves $[C]\in \calM_6$ having a pencil $M\in W^1_4(C)$ with $h^0(C, K_C \otimes M^{\otimes (-2)})\geq 1$. We have the following formula for the class of its closure in $\mm_6$, see \cite{EH87}:
$$
  [\overline{\mathcal{GP}}_{6, 4}^1]=94\lambda-12\delta_0-50\delta_1-78\delta_2-88\delta_3\in CH^1(\mm_6).
$$
$\bullet$ The locus $\mathcal{GP}_{6,5}^1$ of curves with a vanishing theta characteristic; then
$$
  [\overline{\mathcal{GP}}_{6, 5}^1]= 8\left(65\lambda-8\delta_0-31\delta_1-45\delta_2-49\delta_3\right)\in CH^1(\mm_6).
$$

The Prym map $P:\calR_6\rightarrow \calA_5$ is dominant of degree 27 and its Galois group equals the Weyl group of $E_6$, see \cite{dosm},\cite{dofibers}. The differential of the Prym map at the level of stacks
$$(dP)_{(C, \eta)}: H^0(C, K_C^{\otimes 2})^{\vee}\rightarrow \left(\Sym^2 H^0(C, K_C\otimes \eta)\right)^{\vee}$$ is the dual of the multiplication map at the level of global sections for the Prym-canonical map $\varphi_{K_C\otimes \eta}$. Thus the ramification divisor of $P$ is a Cartier divisor on $\calR_6$ supported on the locus
$$\calQ:=\left\{(C, \eta)\in \calR_6:\Sym^2 H^0(C, K_C\otimes \eta)\stackrel{\ncong}\longrightarrow H^0(C, K_C^{\otimes 2})\right\}.$$ The closure of $P(\calQ)$ inside $\calA_5$ is the branch divisor of $P$. At a general point $(A, \Theta) \in P(\calQ)$ the fiber of $P$ has the structure of the set of lines on a one-nodal cubic surface, that is, $P^{-1}(A, \Theta)\cap \calQ$ consists of $6$ ramification points corresponding to the $6$ lines through the node. The remaining 15 points of $P^{-1}(A, \Theta)$ are in correspondence with the $15$ lines on the one-nodal cubic surface not passing through the node. Since $\mathrm{deg}(P)=27$ it follows that $P$ has simple ramification and $\calQ$ is reduced. Donagi \cite[p.~93]{dofibers} established that $\calQ$ is irreducible by showing that the monodromy acts transitively on a general fiber of $P_{| \calQ}$. We sketch a different proof  which uses the irreducibility of the moduli space of polarized Nikulin surfaces. We summarize these results as follows:

\begin{prop}
Set-theoretically, the branch divisor of the map $P$ is equal to the closure $N_0'$ of $P(\calQ)$ in $\calA_5$. At the level of cycles, $P_*[\calQ]=6[N_0']$.
\end{prop}
We turn our attention to the geometry of $\calQ$. First we compute the class of its closure in $\pr_6$, then we link it to Prym-Brill-Noether theory:

\begin{thm}\label{Qclass}
The ramification divisor $\calQ\subset\calR_6$ is irreducible.  The class of its closure $\widetilde{\calQ}$ in $\pr_6$ equals
$$[\widetilde{\calQ}]= 7\lambda-\delta_0^{'}-\frac{3}{2}\delta_0^{\mathrm{ram}}-c_{\delta_0^{''}}\delta_0^{''}\in CH^1(\pr_6),$$
where we have the estimate $c_{\delta_0^{''}}\geq 4.$\end{thm}
\begin{proof}
The irreducibility of $\calQ$ follows from \cite[Theorem 0.5]{faveni}, where it is proved that $\calQ$ can be realized as the image of a projective bundle over the \emph{irreducible} moduli space $\calF_6^{\mathfrak{N}}$ of polarized Nikulin $K3$ surfaces of genus $6$.

\noindent To estimate the class of the closure $\widetilde{\calQ}$ of $\calQ$ in $\pr_6$, we set up two tautological vector bundles $\calN_1$ and $\calN_2$ over $\pr_6$ having fibers
$$\calN_1(X, \eta):=H^0(X, \omega_X\otimes \eta)\quad \mbox{and}\quad \calN_2(X, \eta):=H^0(X, \omega_X^{\otimes 2}\otimes \eta^{\otimes 2})$$
over a point $(X, \eta)\in \pr_6$. There is a morphism $\phi:\Sym^2(\calN_1)\rightarrow \calN_2$ between vector bundles of the same rank given by multiplication of Prym-canonical forms, and we denote by $\mathcal{Z}$ the degeneracy locus of $\phi$. Using \cite[Proposition 1.7]{falu} we have the following formulas in $CH^1(\pr_6)$
$$
  c_1(\calN_1)=\lambda-\frac{1}{4}\delta_0^{\mathrm{ram}}\quad  \mbox{and}\quad c_1(\calN_2)=13\lambda-\delta_0^{'}-\delta_0^{''}-3\delta_0^{\mathrm{ram}},
$$
thus $[\mathcal{Z}]=c_1(\calN_2)-6c_1(\calN_1)=7\lambda-\delta_0^{'}-\delta_0^{''}-\frac{3}{2}\delta_0^{\mathrm{ram}}.$
By definition, $\calQ=\mathcal{Z}\cap \calR_6$. Furthermore, $\phi$ is non-degenerate at a general point of $\Delta_0^{'}$ and $\Delta_0^{\mathrm{ram}}$, hence the difference $\mathcal{Z}-\widetilde{\calQ}$ is an effective divisor supported only on $\Delta_0^{''}$.

Assume now that $(X, \eta)\in \Delta_0^{''}$ is a generic point corresponding to a normalization map $\nu:C\rightarrow X$, where $[C, x, y]\in \calM_{5, 2}$ and $x, y\in C$ are distinct points such that $\nu(x)=\nu(y)$. Since $\nu^*(\eta)=\calO_C$, we obtain an identification
$H^0(X, \omega_X\otimes \eta)=H^0(C, K_C)$ whereas $H^0(X, \omega_X^{\otimes 2}\otimes \eta^{\otimes 2})$ is a codimension one subspace of
$H^0(C, K_C^{\otimes 2}(2x+2y))$ described by a residue condition at $x$ and $y$. It is straightforward to check that the kernel
$$\Ker\phi(X, \eta)=\Ker\left\{\Sym^2 H^0(C, K_C)\rightarrow H^0(C, K_C^{\otimes 2})\right\}$$
has dimension $3$. Thus $[\mathcal{Z}]-[\widetilde{\mathcal{Q}}]- 3\delta_0^{''}$ is effective supported on $\Delta_0^{''}$, which implies that $c_{\delta_0^{''}}\geq 4$.
\end{proof}
\begin{rem} We shall prove later that in fact $c_{\delta_0^{''}}=4$.
\end{rem}

Even though the locus $\calQ$ is defined in terms of syzygies of Prym-canonical curves, its points have a characterization in terms of stable singularities of Prym theta divisors.
\begin{thm}
The theta divisor of a Prym variety $P(C, \eta)\in \calA_5$ has a stable singularity if and only if $P$ ramifies at the point $(C, \eta)$, that is,
$$\calQ=\left\{(C, \eta)\in \calR_6:\Sing_{(C, \eta)}^{\mathrm{st}}(\Xi)\neq \emptyset\right\}. $$
\end{thm}
\begin{proof}
Let us denote by $\mathcal{W}:=\{(C, \eta)\in \calR_6: V_3(C, \eta)\neq \emptyset\}$ the Prym-Brill-Noether locus corresponding to stable singularities of
Prym theta divisors. Our aim is to show that $\mathcal{W}=\calQ$; we begin by establishing the inclusion $\mathcal{W}\subset \calQ$.
First note that if $[C]\in \calM_6$ is trigonal, for any two-torsion point $\eta\in \Pic^0(C)[2]-\{\calO_C\}$ we can write
$K_C\otimes \eta=A\otimes A'$, where $A\in W^1_3(C)$ and $A'\in W^1_7(C)$. This implies that
$(C, \eta)\in \calQ$.

Fix now $(C, \eta)\in \calR_6$ and a line bundle $L\in V_3(C, \eta)$. If $h^0(\tilde{C}, L)\geq 6$, then $\tilde{C}$ (and hence $C$ as well) must be hyperelliptic, so $(C, \eta)\in \calQ$ by the previous remark. We may thus assume that $h^0(\tilde{C}, L)=4$ and consider the associated Pfaffian quadric $Q_L\in \Sym^2 H^0(C, K_C\otimes \eta)$. If $Q_L\neq 0$, then it contains the Prym-canonical model $\varphi_{K_C\otimes \eta}(C)$, in particular $(C, \eta)\in \calQ$. If $Q_L\equiv 0$, then there exists $M\in \Pic(C)$ with $h^0(C, M)\geq 3$ and an effective divisor $D$ on $\tilde{C}$, such that $L=f^*(M)\otimes \calO_{\tilde{C}}(B)$. If $\deg(M)\leq 4$ then $C$ is hyperelliptic, hence $(C, \eta)\in \calQ$. If $\deg(M)=5$, then $B=0$ and $C$ is a smooth plane quintic such that $h^0(C, M\otimes\eta)=1$. It is known, see \cite[Section 4.3]{dofibers},  that in this case $P(C, \eta)$ is the intermediate Jacobian of a cubic threefold and the differential $(dP)_{(C, \eta)}$ has corank $2$, thus once more $(C, \eta)\in \calQ$.
\vskip 3pt

Therefore $\mathcal{W}\subset \calQ$. We claim that $\mathcal{W}$ has at least a divisorial component, which follows by exhibiting a point $(C, \eta)\in \calR_6$ and a line bundle $L\in V_3(C, \eta)$ such that $\mu_L^-$ is surjective. Assuming this for a moment, we conclude that $\mathcal{W}=\calQ$ by invoking the irreducibility of $\calQ$.

\vskip 3pt
To finish the proof we use a realization of Prym curves $(C, \eta)\in \calR_6$ with $V_3(C, \eta)\neq \emptyset$ resembling \cite[Section 2]{faveni}. For a line bundle $L\in V_3(C, \eta)$ with $h^0(\tilde{C}, L)=4$, if $\mu^+_L:\Sym^2 H^0(\tilde{C}, L)\rightarrow H^0(C, K_C)$ denotes the $i^*$-invariant part of the Petri map, one has the following commutative diagram:
$$\xymatrix@R=11pt{\tilde{C} \ar[rr]^-{(L, i^{*} L)} \ar[dd]_f && \PP^3\times \PP^3 \ar[dd]_q \ar@{^{}->}[rd] \\
&& & \PP^{15}=\PP\left(H^0 (L)^{\vee}\otimes H^0(L)^{\vee}\right) \ar@{-->}@<-1ex>[dl] \\
C\ar[rr]^-{\mu_L^+} && *!<-22pt,0pt>{\PP^9=\,\PP(\Sym^2 H^0(L)^{\vee})}}
$$

In this diagram $q:\PP^3\times \PP^3\rightarrow \PP^9$ is the map $a\otimes b\mapsto a\otimes b+b\otimes a$ into the projective space of symmetric tensors. Reversing this construction, if $\iota\in \op{Aut}(\PP^3\times \PP^3)$ denotes the involution interchanging the two factors, the complete intersection of $\PP^3\times \PP^3$ with $4$ \emph{general} $\iota^*$-invariant hyperplanes in $H^0(\mathcal{I}_{\PP^3\times \PP^3}(1, 1))^+$ {}and one \emph{general} $\iota^*$-anti-invariant hyperplane in $H^0(\mathcal{I}_{\PP^3\times \PP^3}(1, 1))^-$ is a smooth curve $\tilde{C}\subset \PP^3\times \PP^3$; the automorphism $\iota_{|C}$ induces a double cover $f:\tilde{C}\rightarrow C$ such that $\Ker(\mu_L^-)$ has $1$-dimensional kernel corresponding to the unique element in $H^0(\mathcal{I}_{\PP^3\times \PP^3}(1, 1))^-$. For more details on this type of argument, we refer to \cite{faveni}.
\end{proof}

\section{The antiramification divisor of the Prym map}
In this section we describe
geometrically the \emph{antiramification divisor} $\calU$ of the map $P:\calR_6\rightarrow \calA_5$, defined via the equality of divisors
\begin{equation}\label{preimD}
 P^*(N_0') = 2\calQ + \calU.
\end{equation}

For a general curve $[C]\in\mathcal{GP}^1_{6, 4}$, {}if $M\in W^1_4(C)$ denotes the pencil such that $\mu_0(M)$ is not injective, we let $x+y\in C_2$ be the support of the unique section of $K_C \otimes M^{\otimes (-2)}${}. We consider the four line bundles
$$
L_{u,v} := f^*M \otimes \calO_{\tilde C}(x_u + y_v)\in \mathrm{Nm}_f^{-1}(K_C),
$$
where $1 \leq u,v \leq 2$ and $f(x_u) = x$, $f(y_v) = y$. Using the parity flipping lemma of \cite{Mumfordprym}, exactly two of the quantities $h^0(\tilde{C}, L_{u, v})$ are equal to $2$, the other being equal to $3$, that is, $\Sing^{\mathrm{ex}}_f(\Xi)$ contains at least two points.
Hence $\pi^*(\mathcal{GP}_{6, 4}^1)\subset \mathcal{U}$. Using Theorem \ref{Qclass}, equality (\ref{preimD}), and the formula for $[\mathcal{GP}_{6, 4}^1]\in CH^1(\calM_6)$, we compute
$$
 [\calU]=P^*([N_0'])-2[\calQ]= 108\lambda-14\lambda= \pi^*([\mathcal{GP}_{6, 4}^1])\in CH^1(\calR_6).
$$

Since the $\lambda$-coefficient of any non-trivial effective divisor class on $\calR_6$ must be strictly positive, we obtain the following result:
\begin{prop} We have the following equality of divisors on $\calR_6$: $$\calU=\pi^*(\mathcal{GP}_{6, 4}^1).$$
\end{prop}

We now determine the pull-back of $\overline{N_0'}$ under the map $P:\pr_6\dashrightarrow \aa_5$. As usual, $\widetilde{\calU}$ denotes the closure of $\calU$ inside $\pr_6$.

\begin{thm}\label{anticlass}
We have the following equality of divisors on $\pr_6$:
$$P^*(\overline{N_0'})= 2\widetilde{\calQ}+\widetilde{\calU}+20\Delta_0^{''}.$$
\end{thm}
\begin{proof}
We use the formula $[\overline{N_0'}]=108\lambda_1-14D$, as well as Theorem \ref{Qclass} and formula (\ref{prympull}) in order to note that the effective class
$P^*(\overline{N_0'})-2\overline{\calQ}-\pi^*(\overline{\mathcal{GP}}_{6, 4}^1)$
is supported only on the boundary divisor $\Delta_0^{''}$.
\vskip 3pt

We now prove that the multiplicity of $\Delta_0^{''}$ in $P^*(\overline{N_0'})$ equals $20$, or equivalently $\op{mult}_{\Delta_0^{''}}(P^*(\overline{N_0}))=40$, since $P(\Delta_0^{''})=\mathcal{J}_5\nsubseteq \theta_{\mathrm{null}}$. Let $\widetilde{\calA}_5:=\mbox{Bl}_{\mathcal{J}_5}(\calA_5)$ be the
blowup of $\calA_5$ along the Jacobian locus and denote by $\calE\subset \widetilde{\calA_5}$ the exceptional divisor. Then $\calE$ is a $\PP^2$-bundle over $\calJ_5$ with the fiber over a point $(\Jac(C), \Theta_C)\in \calJ_5$ being identified with the space $\PP(I_2(K_C)^{\vee})$ of pencils of quadrics containing the canonical curve $C\subset \PP^4$. One can lift the Prym map to a map $\tilde{P}:\pr_6\dashrightarrow \widetilde{\calA}_5$ by setting for a general point $(C_{xy}, \eta)\in \Delta_0^{''}$
$$\tilde{P}(C_{xy}, \eta):=\left((\Jac(C), \Theta_C), q_{xy}\right)\in \widetilde{\calA}_5,$$
where $q_{xy}\in \PP(I_2(K_C)^{\vee})$ is the pencil of quadrics containing the union $C\cup \langle x, y\rangle\subset \PP^4$.
Furthermore, $\tilde{P}^*(\calE)=\Delta_0^{\mathrm{ram}}$, showing that
$$\op{mult}_{\Delta_0^{''}}P^*(\overline{N_0})=\op{mult}_{\calJ_5}(N_0).$$
To estimate the latter multiplicity we consider a general one-parameter family
$j:U\rightarrow \calA_5$ from a disc $U\ni 0$ such that $j(0)=(\Jac(C), \Theta_C)$, with $[C]\in \calM_5$ being a general curve.
Let $\Theta_U:=U\times_{\calA_5} \Theta\rightarrow U$ be the relative theta divisor over $U$. The image of the differential $(dj)_0(T_0(U))$ can be viewed as a hyperplane $h\subset \PP\left(\Sym^2 H^0(K_C)\right)$. The variety $\Theta_U$ has ordinary double points at those points
$(0, L)\in \Theta_U$ where $L\in \Sing(\Theta_C)=W^1_4(C)$ is a singularity such that its tangent cone $Q_L\in \PP I_2(K_C)$ belongs to $h$.
Since the assignment $W^1_4(L)\ni L\mapsto Q_L\in \PP I_2(K_C)$ is an unramified double cover over a smooth plane quintic, we find that
$\Theta_U$ has $10$ nodes. Using the theory of Milnor numbers for theta divisors as explained in \cite{smva5} we obtain
$$
  \op{mult}_{\mathcal{J}_5}(N_0)=\chi(\theta_{\mathrm{gen}})-\chi(W_4(C))+10,
$$
where $\chi(\theta_{\mathrm{gen}})=5!=120$ is the topological Euler characteristic of a general (smooth) theta divisor of genus $5$.
We finally determine $\chi(W_4(C))$, using the resolution $C_4\rightarrow W_4(C)$. From the Macdonald formula, see \cite{ACGH}, $\chi(C_4)=(-1)^{g-1}{2g-2\choose g-1}_{| g=5}=70$, whereas $\chi(W^1_4(C))=-20$, because $g(W^1_4(C))=11$. Therefore
$\chi(C^1_4)=2\chi(W^1_4(C))=-40$. We find that
$\chi(W_4(C))=\chi(C_4)-\chi(C_4^1)+\chi(W^1_4(C))=90$, thus $\op{mult}_{\mathcal{J}_5}(N_0)=120-90+10=40.$
\end{proof}

\begin{cor}\label{qclcom}
We have the following formula in $CH^1(\pr_6)$:{}
$$[\widetilde{\calQ}]=7\lambda-\delta_0^{'}-4\delta_0^{''}-\frac{3}{2}\delta_0^{\mathrm{ram}}.$$
\end{cor}
We exploit the geometry of the ramification and antiramification divisors of the Prym map and determine the pushforward of divisor classes on $\rr_6$:
\begin{thm}\label{class2}
The pushforwards of tautological divisor classes via the rational Prym map $P:\rr_6\dashrightarrow \aa_5$ are as follows:
$$
  P_*(\lambda)= 18\cdot 27 \lambda_1 -57D,\quad P_*(\delta_0^{\mathrm{ram}})=4(17\cdot 27\lambda_1 -57D),
$$
$$
  P_*(\delta_0^{'})=27D,\quad P_*(\delta_0^{''})=P_*(\delta_i)=P_*(\delta_{i:g-i})=0\qquad
  \mbox{for}\ 1\leq i\leq g-1.$$
\end{thm}
We point out that even though $P$ is not a regular map, it can be extended in codimension $1$ such that $P$ is the morphism induced at the level of coarse moduli spaces by a proper morphism of stacks, see e.g. \cite{dofibers} p.63-64. Furthermore $P^{-1}$ contracts no divisors, in particular, we can pushforward divisors under $P$ and use the push-pull formula. Perhaps the most novel aspect of Theorem \ref{class2} is the calculation of the class of the divisor $P_*(\Delta_0^{\mathrm{ram}})$ consisting of Prym varieties corresponding to {}ramified double covers $\widetilde{C}\rightarrow C$ of genus $5$ curves with two branch points{}.
\begin{proof}
We write the following formulas in $CH^1(\aa_5)$:
$$
  27\lambda_1 =P_*P^*\lambda_1=P_*(\lambda)-\frac{1}{4}P_*(\delta_0^{\mathrm{ram}}),
$$
$$
 6\cdot(108 \lambda_1 -14D)=6[\overline{N_0'}]=P_*([\overline{\calQ}])=7P_*(\lambda)-\frac{3}{2}P_*(\delta_0^{\mathrm{ram}})-P_*(\delta_0^{'}).
$$
From \cite{grsmprym} it follows that $P_*(\delta_0^{'})=27D$, whereas obviously $P_*(\delta_0^{''})=0$, which suffices to solve the system of equations for coefficients.
\end{proof}
\section{The slope of $\aa_5$}
Using the techniques developed in previous chapters, we determine the slope of the perfect cone compactification $\aa_5$ of $\calA_5$ (note also that by the appendix by K.~Hulek to \cite{grsmprym} this slope is the same for all toroidal compactification).
We begin with some preliminaries. Let $D$ be a ${\bf Q}$-divisor on a normal ${\bf Q}$-factorial variety $X$. We say that $D$ is rigid if $|mD|=\{mD\}$ for all sufficiently large and divisible integers $m$. Equivalently, the Kodaira-Iitaka dimension $\kappa(X, D)$ equals zero.

We denote by $\BB(D):=\bigcap_{m} \mbox{Bs}(|mD|)$ the \emph{stable base locus} of $D$. We say that $D$ is \emph{movable} if $\op{codim } \BB(D)\geq 2$.

Recall that one defines the slope of $\aa_g$ as
$s(\aa_g):=\inf_{E\in \mathrm{Eff}(\aa_g)} s(E)$. In a similar fashion one defines the \emph{moving slope} of $\aa_g$ as the slope of the cone of moving divisors on $\aa_g$, that is,
$$s'(\aa_g):=\inf\left\{s(E):E\in \mathrm{Eff}(\aa_g),\ E \mbox{ is movable}\right\}.$$
Thus $s'(\aa_g)$ measures the minimal slope of a divisor responsible for a non-trivial map from $\aa_g$ to a projective variety.
{}It is known that $s(\aa_4)=8$ \cite{smmodfour}, and as an immediate consequence of the result  about the slope of $\calM_4$ we have that $s'(\aa_4)=s(\overline{\theta}_{\mathrm{null}})=\frac{17}{2}$.
In the next case, that of dimension $g=5$, the formula $[\overline{N_0'}]= 108\lambda_1 -14D$ yields the upper bound $s(\aa_5)\leq \frac{54}{7}$. A lower bound for the slope $s(\aa_5)$ was recently obtained in \cite{grsmprym}.{}
\vskip 4pt

We shall now prove Theorem \ref{slopea5} and establish that
$$\kappa(\aa_5, \overline{N_0'})<\mbox{dim}(\aa_5),$$
in particular showing that $s(\aa_5)=\frac{54}{7}$. To prove Theorem \ref{slopea5} we translate the problem into a question on the linear series $|P^*(\overline{N_0'})|$ on $\rr_6$.
One can show that each of the components of $P^*(\overline{N_0'})$ is an extremal divisor on $\rr_6$, however their sum could well have positive Kodaira dimension. Of crucial importance is a uniruled parametrization of $\qq$ using sextics with a totally tangent conic.

\vskip 4pt

We fix general points $q_1, \ldots, q_4\in \PP^2$, then set $S:=\mbox{Bl}_{\{q_i\}_{i=1}^4}(\PP^2)\rightarrow \PP^2$ and denote by $\{E_{q_i}\}_{i=1}^4$ the corresponding exceptional divisors. We make the identification $\PP^{15}:=\left|\calO_{S}(6)(-2\sum_{i=1}^4 E_{q_i})\right|$, then consider the space of $4$-nodal sextics having a totally tangent conic
$$\calX:=\left\{(\Gamma, Q)\in \PP^{15}\times |\calO_{S}(2)|: \Gamma\cdot Q=2d, \mbox{ where } d\in (\Gamma_{\mathrm{reg}})_6\right\}.$$
A parameter count shows that $\calX$ is pure of dimension $14$. We define the rational map
$v:\calX\dashrightarrow \calR_6$
$$v(\Gamma, Q):=\bigl(C, \eta:=\nu^*\left(\calO_{\Gamma}(1)(-d)\right)\bigr)\in \calR_6,$$
where $\nu:C\rightarrow \Gamma$ is the normalization map. The image $v(\calX)$ is expected to be a divisor on $\calR_6$, and we show that this is indeed the case --- this construction yields another geometric characterization of points in $\calQ$.

\begin{thm}\label{qparametrisierung}
 The closure of $v(\calX)$ inside $\calR_6$ is equal to $\calQ$, that is, a general Prym curve $(C, \eta)\in \calQ$ possesses a totally tangent conic.
\end{thm}
\begin{proof} We carry out a class calculation on $\rr_6$ and the result will be a consequence of the extremality properties of
the class $[\qq]\in \mathrm{Eff}(\rr_6)$. We work on a partial compactification $\calR_6'$ of $\calR_6$ that is even smaller than $\pr_6$.
\vskip 2pt

Let $\calR_6':=\calR_6^0\cup \pi^{-1}(\Delta_0^*)$ be the open subvariety of $\rr_6$, where $\calR_6^0$  consists of smooth Prym curves $(C, \eta)$ for which
$\dim W^2_6(C)=0$ and $h^0(C, L\otimes \eta)=1$ for every $L\in W^2_6(C)$, whereas $\Delta_0^*\subset \Delta_0$ is the locus of curves
$[C_{xy}]$, where $[C]\in \calM_5-\calM_{5, 3}^1$ and $x, y\in C$. Observe that $\op{codim}(\pr_6-\calR_6', \pr_6)=2$, in particular we can identify $CH^1(\calR_6')$ and $CH^1(\pr_6)$.  Over the Deligne-Mumford stack $\textbf{R}_6'$ of Prym curves coarsely represented by the scheme $\calR_6'$ (observe that $\textbf{R}'_6$ is an open substack of $\textbf{R}_6$), we consider the finite cover
$$\sigma:\mathfrak{G}^2_6 \rightarrow \textbf{R}_6',$$ where $\mathfrak G^2_6$ is the Deligne-Mumford stack that classifies triples
$(C, \eta, L)$, with $(C, \eta)\in \calR_6'$ and $L\in W^2_6(C)$. Note that a curve $[C_{xy}]\in \Delta_0'$ carries no non-locally free sheaves $F\in \overline{\Pic^6}(C_{xy})$ with $h^0(C_{xy}, F)\geq 5$, for $F$ would correspond to a $\mathfrak g^2_5$ on the normalization $C$ of $C_{xy}$, a contradiction. The universal curve $p:\calC\rightarrow \mathfrak G^2_6$ is equipped both with a
universal Prym bundle $\mathcal{P}\in \Pic(\calC)$ and a universal Poincar\'e line bundle $\calL\in \Pic(\calC)$ such that
$\calL_{| p^{-1}(C, \eta, L)}=L$, for any $(C, \eta, L)\in \mathfrak G^2_6$.
We form the codimension $1$ tautological classes
\begin{equation}\label{tautological}
\mathfrak{a}:=p_*\left(c_1(\calL)^2\right),\quad\mathfrak{b}:=p_*\left(c_1(\calL)\cdot
c_1(\omega_p)\right)\in CH^1(\mathfrak G^2_6),
\end{equation}
and the sheaves $\mathcal{V}_i:=p_*(\calL^{\otimes i})$, where $i=1, 2$. Both $\mathcal{V}_1$ and $\mathcal{V}_2$ are locally free.
The dependence of $\mathfrak a$ and  $\mathfrak b$ on the choice of $\calL$ is discussed in \cite{falu}.
Using the isomorphism $CH^1(\textbf{R}_6')=CH^1(\calR_6')$, one can write the following formulas in $CH^1(\calR_6')$, see \cite[page 776]{falu}:
\begin{equation}\label{tau}
\sigma_*(\mathfrak a)=-48\lambda+7\pi^*(\delta_0),\ \sigma_*(\mathfrak b)=36\lambda-3\pi^*(\delta_0),\ \sigma_*(c_1(\mathcal{V}))=-22\lambda+3\pi^*(\delta_0).
\end{equation}
\vskip 3pt
We also introduce the sheaf $\calE:=p_*(\mathcal{P}\otimes \calL)$. Since $R^1p_*(\mathcal{P}\otimes \calL)=0$ (this is the point where we use $H^1(C, L\otimes \eta)=0$ for each $(C, \eta, L)\in \calR_6'$), applying Grauert's theorem we obtain that $\calE$ is locally free and
via Grothendieck-Riemann-Roch we compute its Chern classes. Taking into account that $p_*(c_1(\mathcal{P})^2)=\delta_0^{\mathrm{ram}}/2$ and $p_*(c_1(\calL)\cdot c_1(\mathcal{P}))=0$, see \cite[Proposition 1.6]{falu}, one computes
\begin{equation}\label{tau2}
c_1(\calE)=\lambda-\frac{\delta_0^{\mathrm{ram}}}{4}+\frac{\mathfrak{a}}{2}-\frac{\mathfrak{b}}{2}\in CH^1(\mathfrak G^2_6).
\end{equation}
Similarly, by GRR we find that $c_1(\mathcal{V}_2)=\lambda-\mathfrak{b}+2\mathfrak{a}$.

\vskip 3pt
After this preparation we return to the problem of describing the closure $\widetilde{v(\calX)}$ of $v(\calX)$ in $\calR_6'$. For a point
$(C, \eta, L)\in \mathfrak{G}^2_6$, the {}two-torsion point{} $\eta$ is induced by a conic totally tangent to the image of
$\nu:C\stackrel{|L|}\rightarrow \Gamma\subset \PP^2$, if and only if the map given by multiplication followed by projection
$$\chi(C, \eta, L):H^0(C, L\otimes \eta)\otimes H^0(C, L\otimes \eta)\rightarrow H^0(C, L^{\otimes 2})/\mathrm{Sym}^2 H^0(C, L)$$
is not an isomorphism. Working over the stack we obtain a morphism of vector bundles over $\mathfrak G^2_6$
$$\chi:\calE^{\otimes 2}\rightarrow \mathcal{V}_2/\mathrm{Sym}^2(\mathcal{V}_1),$$
such that the class of $\widetilde{v(\calX)}$ is (up to multiplicity) equal to
$$\sigma_* c_1\Bigl(\frac{\mathcal{V}_2}{\mathrm{Sym}^2(\mathcal{V}_1)}-\calE^{\otimes 2}\Bigr)=35\lambda-5(\delta_0^{'}+\delta_0^{''})-\frac{15}{2}\delta_0^{\mathrm{ram}}=5[\mathcal Z],$$
where we have used both (\ref{tau}) and (\ref{tau2}). We recall that the cycle $\mathcal Z$ was defined in the proof of Theorem \ref{Qclass} as a subvariety of the larger space $\pr_6$ with the property
that $\mathcal Z\cap \calR_6'=\calQ\cap \calR_6'$. Thus the class $[\widetilde{v(\calX)}]\in CH^1(\calR_6')$ is proportional (up to the divisor class $\delta_0^{''}$) to the class $[\widetilde{\calQ}]$. It is proved in \cite[Proposition 3.6]{faveni}  that if $D$ is an effective divisor on $\calR_6$ such that $[\overline{D}]=\alpha[\overline{\calQ}]+\beta\ \delta_0^{''}$, then one has the set-theoretic equality $D=\calQ$. Thus we conclude that the closure of ${v(\calX)}$ in $\calR_6$ is precisely $\calQ$.
\end{proof}

\begin{thm}\label{pencil1}\footnote{Added in April 2022: In the published version, the intersection numbers $R\cdot \delta_0^{\mathrm{ram}}$ and $R\cdot \delta_0^{'}$ are computed incorrectly. This is corrected here, with every detail provided. These changed intersection numbers are what lead to a slightly weaker version of Theorem 0.7 than in the published version, as per the following discussion.}
Through a general point of the ramification divisor $\overline{\calQ}$ there passes a rational curve $R\subset \rr_6$ with the following numerical features:
$$R\cdot \lambda=6, \ R\cdot \delta_0^{'}=27, \ R\cdot \delta_0^{''}=0, \ R\cdot \delta_0^{\mathrm{ram}}=10,
\ R\cdot \delta_i=R\cdot \delta_{i: 5-i}=0,$$ for $i=1, \ldots, 4$. In particular $R\cdot \qq=0$ and $R\cdot \uu=0$.
\end{thm}
Assuming for the moment Theorem \ref{pencil1}, we explain how it implies Theorem \ref{slopea5}. Assume that $E\in\mathrm{Eff}(\aa_5)$ with $s(E)\leq s(\overline{N_0'})$. First note that one can assume that $\overline{N_0'}\nsubseteq \op{supp}(E)$, for else, we can replace $E$ by an effective divisor of the form $E':=E-\alpha \overline{N_0'}$ with $\alpha>0$ and still $s(E')\leq s(\overline{N_0'})$. After rescaling by a positive factor, we can write
$E\equiv \overline{N_0'}-\epsilon \lambda_1 \in \mathrm{Eff}(\aa_5)$, where $\epsilon\geq 0$. Clearly we have $P^*(E)\in \mathrm{Eff}(\rr_6)$; observe that since $\overline{N_0'}$ is not a component of $E$, the ramification divisor $\qq$ cannot be a component of $P^*(E)$ either. Thus $R\cdot P^*(E)\geq 0$, that is,
$$
  0\leq R\cdot P^*(E)=R\cdot (2\qq+\uu+20\delta_0^{''})-\epsilon R\cdot \left(\lambda-\frac{\delta_0^{\mathrm{ram}}}{4}\right)=-\frac{7\epsilon}{2},
$$
which implies $\epsilon=0$. Thus $s(E)=s(\overline{N_0'})$ and $E$ cannot be a big divisor.

\begin{proof}[Proof of Theorem \ref{pencil1}] We retain the notation from Theorem \ref{qparametrisierung} and fix a general element $(C, \eta)\in \calQ$ corresponding to a sextic curve $\Gamma\subset \PP^2$ having nodes at $q_1, \ldots, q_4$. From Theorem \ref{qparametrisierung} we may assume that there exists a conic $Q\subset \PP^2$ such that $Q\cdot \Gamma=2(p_1+\cdots+p_6)$, where $p_1, \ldots, p_6\in \Gamma_{\mathrm{reg}}$. Since the points $q_1, \ldots, q_4\in \PP^2$ are distinct and no three are collinear, it follows that $[C]\notin \overline{\mathcal{GP}}_{6, 4}^1$, and this holds even when $C$ has nodal singularities.

To construct the pencil $R\subset \rr_6$, we reverse this construction and start with a conic $Q\subset \PP^2$ and six general points $p_1, \ldots, p_6\in Q$ on it. On the blowup $S'$ of $\PP^2$ at the $10$ points $q_1, \ldots, q_4, p_1, \ldots, p_6$, we denote by $\{E_{p_i}\}_{i=1}^6$ and by $\{E_{q_j}\}_{j=1}^4$  the exceptional divisors. For $1\leq i\leq 6$, let $l_i\in E_{p_i}$ be the point corresponding to the tangent line $T_{p_i}(Q)$. If $\tilde{S}$ is the blowup of $S'$ at $l_1 , \ldots, l_6$, by slight abuse of notation we denote by $E_{p_i}, E_{l_i}$ and $E_{q_j}$ the exceptional divisors on $\tilde{S}$ (respectively the proper transforms of exceptional divisors on $S'$).
Then $\dim\left|\calO_{\tilde{S}}(6)\left(-2\sum_{j=1}^4 E_{q_j}-\sum_{i=1}^6 (E_{p_i}+E_{l_i})\right)\right|=3$, and we choose a general pencil in this linear system. This pencil induces a curve $R\subset \rr_6$. Note that the pencil contains one distinguished element $t_0$, consisting of the union of $Q$ and two conics $Q_1$ and $Q_2$ passing through $q_1, \ldots, q_4$.

Considering the pushforward $\pi_*(R)\subset \mm_6$, after a routine calculation we find:
$$
 R\cdot \lambda=6,\ R\cdot (\delta_0^{'}+\delta_0^{''}+2\delta_0^{\mathrm{ram}})=\pi_*(R)\cdot \delta_0=47, \ \pi_*(R)\cdot \delta_i=0\ \mbox{ for } i=1, 2, 3.
$$
In particular, as expected $R\cdot \uu=0$. The points in $R\cap \Delta_0^{\mathrm{ram}}$ correspond to the case when the underlying Prym structure is not locally free, which happens when one of the points $p_i$ becomes singular. For each $1\leq i\leq 6$, there is one such curve in $R$.

\vskip 3pt

We are left with the task of determining the multiplicity $(R\cdot \delta_0^{\mathrm{ram}})_{t_0}$. To that end we describe in more detail the Prym structure of the curves in the pencil $R$. We denote by $\varphi: \PP^1\times \PP^1\rightarrow \PP^2$ the double cover branched along the conic $Q$. For a sextic curve $\Gamma\subset \PP^2$ nodal at $q_1, \ldots, q_4$ and with $\Gamma\cdot Q=2(p_1+\cdots+p_6)$, we observe that $\Gamma':=\varphi^{-1}(\Gamma)$ has nodes at the points in $\varphi^{-1}(q_1)\cup \ldots \cup  \varphi^{-1}(q_4)$, as well as at $\varphi^{-1}(p_i)$, for $i=1, \ldots, 6$. To $\Gamma$ we associate the \'etale double cover $C'\rightarrow C$, where $C$ is the normalization of $\Gamma$ and $C'$ is the normalization of $\Gamma'$. Assume now $\Gamma_0:=Q\cup Q_1\cup Q_2$ is the curve corresponding to the point $t_0\in R$ and set $Q\cap Q_1=\{z_1, \ldots, z_4\}$ and $Q\cap Q_2=\{z_1', \ldots, z_4'\}$. We denote by $C_0$ the partial normalization of $\Gamma_0$ at the $4$ points of intersection $Q_1\cdot Q_2$. Applying stable reduction, the associated double cover $C'_0\rightarrow C_0$ has as source curve $C_0'$ the union $Q'\cup Q_1'\cup Q_2'$, where $Q_i'$ is the double cover of $Q_i$ ramified over the $4$ points in $Q\cdot Q_i$, whereas $Q'$ is the hyperelliptic genus $3$ cover of $Q$ ramified over $z_1, z_1', \ldots, z_4, z_4'$. Note that $Q_1'$ and $Q_2'$ are disjoint, hence $p_a(C_0')=g(Q')+g(Q_1')+g(Q_2')+2\cdot(4-1)=11$. Since $C_0'\rightarrow C_0$ is ramified over each of the eight nodes $z_1, z_1', \ldots, z_4, z_4'$ of $C_0$, using for instance \cite[Remark 1.2]{falu}, it follows that the Prym curve associated to this cover is of the form $[X, \eta, \beta]$, where $X$ is the quasi-stable curve obtained from $C_0$ by inserting smooth rational components $E_1, E_1', \ldots, E_4, E_4'$ at the points $z_1, z_1', \ldots, z_4, z_4'$ respectively, the line bundle $\eta\in \mbox{Pic}^0(X)$ satisfies $\eta_{E_i}=\mathcal{O}_{E_i}(1)$ and $\eta_{E_i'}=\mathcal{O}_{E_i'}(1)$, whereas if $X':=\overline{X\setminus \bigcup_{i=1}^4 (E_i\cup E_i')}$, then
$$\eta^{\otimes 2}=\mathcal{O}_{X'}\Bigl(-\sum_{i=1}^4 (x_i+y_i+x_i'+y_i')\Bigr),$$ where
 $\{x_i, y_i\}=E_i\cap \overline{X\setminus E_i}$ and $\{x_i', y_i'\}=E_i'\cap \overline{X\setminus E_i'}$ for $i=1, \ldots, 4$.

\vskip 3pt

We denote by $\mathbb C_{\tau}^{3g-3}$ the versal deformation space of $[X, \eta, \beta]\in \overline{\mathcal{R}}_6$ and choose local coordinates 
$\tau_1, \ldots, \tau_{3g-3}$ such that for $i=1, \ldots, 8$, the hyperplane $(\tau_i=0)$ corresponds to the locus where the exceptional component $E_i$ persists. Furthermore, if $\mathbb C_t^{3g-3}$ is the versal deformation space of the stable model $C_0$ of $X$ and $\mathcal{C}\rightarrow \mathbb C^{3g-3}_t$ is the universal family, we consider the map
$\mathbb C^{3g-3}_{\tau}\rightarrow \mathbb C_t^{3g-3}$ given by $t_i=\tau_i^2$ for $i=1, \ldots, 8$ and $t_i=\tau_i$ for $9\leq i\leq 3g-3=15$. Then the universal family $\mathcal{X}\rightarrow \mathbb C_{\tau}^{3g-3}$ of Prym curves of genus $6$ is obtained from the fibre product
$\mathcal{C}':=\mathcal{C}\times_{\mathbb C_t^{3g-3}} \mathbb C_{\tau}^{3g-3}$ by blowing-up the codimension two loci corresponding to the sections $(\tau_i=0)\rightarrow \mathcal{C}'$. It follows that the intersection multiplicity of $R\times_{\overline{\mathcal{R}}_6}\mathbb C_{\tau}^{3g-3}$ with the locus $(\tau_1 \cdots \tau_8)=0$ is equal to $8$ and accordingly 
$$(R\cdot \delta_0^{\mathrm{ram}})_{t_0}=\frac{1}{2}\Bigl(R\cdot \pi^*(\delta_0)\Bigr)_{t_0}=\frac{1}{2}\Bigl(\pi_*(R)\cdot \delta_0\Bigr)_{[C_0]}=\frac{8}{2}=4,$$ therefore
$R\cdot \delta_0^{\mathrm{ram}}=6+4=10$. Observing that $t_0\in R$ is the only point in the pencil corresponding to a reducible double cover, we also conclude that $R\cdot \delta_0^{''}=0$, thus $R\cdot \delta_0^{'}=27$. Using Corollary \ref{qclcom}, we conclude that
$$R\cdot \qq=7R\cdot \lambda-R\cdot \delta_0^{'}-\frac{3}{2}R\cdot \delta_0^{\mathrm{ram}}=42-27-15=0.$$
\end{proof}




\begin{rem}
Theorem \ref{slopea5} implies that the divisor $\overline{N_0'}$ can be contracted via a birational map having $\aa_5$ as its source. Especially from the point of view of the Minimal Model Program for $\aa_5$, it would be interesting to find a new compactification of the moduli space of ppav $\calA_5^*$, and a birational map $f:\aa_5\dashrightarrow \calA_5^*$ such that $f$ contracts $\overline{N_0'}$.
\end{rem}

\section{The Prym realization of the components of $H$}

For each irreducible component of $H=N_0^{' 4}$ in $\calA_5$, we describe an explicit codimension $2$ subvariety of $\calR_6$ which dominates it
via the Prym map. As a consequence, we prove that $H$ consists of two irreducible components, both unirational and of dimension $13$. We define two subvarieties of $\calR_6$ corresponding to Prym curves $(C, \eta)$ such that $\varphi_{K_C\otimes \eta}$ lies on a quadric of rank at most $4$, cutting a (Petri special) pencil on $C$. Depending on the degree of this pencil, we denote these loci by $\calQ_4$ and $\calQ_5$ respectively.

\begin{df}
We denote by $\calQ_5$ the closure in $\calR_6$ of the locus of curves $(C, \eta)\in \calR_6$ such that $C$ carries two vanishing theta characteristics $\theta_1, \theta_2\in W^1_5(C)$ with $\eta=\theta_1\otimes \theta_2^{\vee}$.
\end{df}
\noindent Equivalently,  $K_C\otimes \eta=\theta_1\otimes \theta_2$, which implies that the Prym-canonical model of $C$ lies on a quadric $Q\subset \PP^4$ of rank $4$, whose rulings induce $\theta_1$ and $\theta_2$ respectively.

\begin{df}
We denote by $\calQ_4$ the closure in $\calR_6$ of the locus of curves $(C, \eta)\in \calR_6$ such that $\eta\in W^1_4(C)-W_4(C)$ and $K_C\otimes \eta$ is very ample.
\end{df}
Equivalently, $K_C\otimes \eta=A\otimes A'$, where $A\in W^1_4(C)$ and $A'\in W^1_6(C)$, and then the image $\varphi_{K_C\otimes \eta}(C)$ lies on a quadric $Q\subset \PP^4$ of rank at most $4$, whose rulings cut out $A$ and $A'$ respectively.

\begin{rem}
Along the same lines, one can consider the locus $\calQ_3$ of curves $(C, \eta)\in \calR_6$ such that $K_C\otimes \eta=A\otimes A'$, where
$A\in W^1_3(C)$ and $A'\in W^1_7(C)$. Observe that $\calQ_3=\pi^{-1}(\calM_{6, 3}^1)$, where $\calM_{6, 3}^1$ is the trigonal locus inside $\calM_6$. In particular, $\op{codim}(\calQ_3, \calR_6)=2$. However from the trigonal construction \cite[Section 2.4]{dofibers}, it follows that $P(\calQ_3)=\calJ_5$, that is, $P$ blows-down $\calQ_3$ and thus $\calQ_3$ plays no further role in describing the components of $H$ in $\calA_5$.
\end{rem}

First we show that $\calQ_4$ lies both in the ramification and the antiramification divisor of the Prym map:

\begin {prop}
$\calQ_4 \subseteq \calQ \cap \calU$.
\end{prop}
\begin{proof} We choose a point $(C, \eta)\in \calQ_4$ general in a component of $\calQ_4$ and write $\eta=M\otimes \calO_C(-D)$, where $M\in W^1_4(C)$ and $D\in C_4$ is an effective divisor. Then we compute
$$
h^0(C, K_C\otimes\eta(-D)) = h^0(K_C\otimes M^{\vee}) = 3,
$$
that is, $\ell:=\langle D\rangle$ is a 4-secant line to the Prym canonical model $\varphi_{K_C\otimes \eta}(C)$. Moreover $\ell$ is contained in the rank $4$ quadric $Q$ whose rulings cut out on $C$ the pencils $M$ and $K_C\otimes \eta\otimes M^{\vee}$ respectively.  The line $\ell$  is not contained in a plane of $Q$ belonging to
the ruling $\Lambda$ that cuts out on $C$ the pencil $M$, for else it would follow that $\eta=0$. Then  $\ell$ is unisecant to the planes in $\Lambda$ and if $d_M\in |M|$ is a general element,  then
$\langle D+d_M\rangle$
is a hyperplane in $\PP^4$. Thus
$$K_C\otimes \eta=\calO_C(d_M+D+x+y),$$
where $x, y\in C$, that is, $H^0(C, K_C\otimes M^{\otimes (-2)})\neq 0$, and $[C]\in \mathcal{GP}_{6, 4}^1$.
 \end{proof}
\begin{prop}
The locus $\calQ_4$ is unirational and of dimension 13.
\end{prop}
\begin{proof}
Since $\calQ_4\subset \calQ\cap \mathcal{U}$, we use the fact that \emph{every} curve $[C]\in \mathcal{GP}^1_{6, 4}$ is a quadratic section of a nodal quintic del Pezzo surface. In the course of proving Theorem \ref{pencil1} we observed that a general Prym curve $(C, \eta)\in \calU$ is characterized by the existence of a totally tangent conic. We show that a similar description carries over to the case of $1$-nodal del Pezzo surfaces.

We fix collinear points $q_1, q_2, q_3\in \PP^2$, a general point $q_4\in \PP^2$, then denote by $\ell:=\langle q_1, q_2, q_3\rangle\subset \PP^2$, by $S':=\mbox{Bl}_{\{q_i\}_{i=1}^4}(\PP^2)\rightarrow \PP^2$ the surface whose image by the linear system $\left|\calO_{S'}(2)\bigl(-\sum_{i=1}^4 E_{q_i}\bigr)\right|$ is a $1$-nodal del Pezzo quintic. Set $\PP^3_{S'}:=\left|\calO_{S'}(3)\bigl(-\sum_{i=1}^3 E_{q_i}-2E_{{q_4}}\bigr)\right|$. Note that $\op{Aut}(S')=\bf C^*$. We consider the $10$-dimensional rational variety
$$\mathcal{V}:=\{(Q, p_1, \ldots, p_5): Q\in |\calO_{\PP^2}(2)|,\ p_1, \ldots, p_5\in Q\}$$
and the rational map $p:\mathcal{V}\dashrightarrow \PP^2$ given by $p((Q, p_1, \ldots, p_5)):=p_6$, where $p_6$ is the residual point of intersection of $Q$ with the unique cubic $E\in |\calO_{\PP^2}(3)|$ passing through $q_1, \ldots, q_4, p_1, \ldots, p_5$. We consider the linear system $$\PP_{(Q, p_1, \ldots, p_5)}:=\Bigr\{\Gamma\in \Bigl|\calO_{S'}(6)\bigl(-2\sum_{i=1}^4 E_{q_i}\bigr)\Bigr|: \Gamma\cdot Q=2(p_1+\cdots+p_6)\Bigr\}.$$
\vskip 3pt

\noindent {\bf{Claim:}} For a general $(Q, p_1, \ldots, p_5)\in \mathcal{V}$, we have  $\dim\ \PP_{(Q, p_1, \ldots, p_5)}=4$, that is, the points $q_1, \ldots, q_4, p_1, \ldots, p_6$ fail to impose one independent condition on $4$-nodal sextic curves.
\vskip 3pt

Since $\ell+Q+\PP^3_{S'}\subset \PP_{(Q, p_1, \ldots, p_5)}$, to conclude that $\dim \PP_{(Q, p_1, \ldots, p_5)}\geq 4$, it suffices to find \emph{one} curve $\Gamma\in \PP_{(Q, p_1, \ldots, p_5)}$ that does not have $\ell$ as a component. We choose $(Q, p_1, \ldots, p_5)\in \mathcal{V}$ general enough that the corresponding cubic $E$ is smooth. Then $2E\in \PP_{(Q, p_1, \ldots, p_5)}$ and obviously $\ell\nsubseteq 2E$.
To finish the proof of the claim, we exhibit a point $(Q, p_1, \ldots, p_5)\in \mathcal{V}$ such that $\dim(\PP_{(Q, p_1, \ldots, p_5)})=4$.
We specialize to the case $p_1\in \ell$ and let $Q_2$ be the conic determined by $p_2, \ldots, p_5$ and $q_4$. The cubic $E$ must equal $\ell+Q_2$ and $p_6\in \ell\cap Q$, so that $E \cdot Q=p_1+\cdots+p_6$. Then $\PP_{(Q, p_1, \ldots, p_5)}=\ell+\PP'$, where
$$\PP':=\Bigl\{Y\in \left|\calO_{S'}(5)\bigl(-\sum_{i=1}^3 E_{q_i}-2E_{q_4}\bigr)\right|: Y\cdot Q_2=p_1+p_6+2(p_2+p_3+p_4+p_5)\Bigr\}.$$
Because $p_2, \ldots, p_5\in \PP^2$ are general, $\dim(\PP')=20-3-3-2-8=4$, which completes the proof of the claim.

\vskip 3pt
We now consider the $\PP^4$-bundle $\mathcal{P}:=\left\{(Q, p_1, \ldots, p_5, \Gamma): \Gamma \in \PP_{(Q, p_1, \ldots, p_5)}\right\}$, together with the map
$u:\mathcal{P}\dashrightarrow \calR_6$, given by
$$u((Q, p_1, \ldots, p_5, \Gamma)):=\bigl(C, \eta:=\calO_C(1)(-p_1-\cdots-p_6)\bigr),$$
where $C\subset S'$ is the normalization of $\Gamma$. Then $M:=\calO_C(2)(-\sum_{i=1}^4E_{q_i})\in W^1_4(C)$ is Petri special and
$|M\otimes \eta|\cong |\calO_{S'}(3)(-\sum_{i=1}^4 E_{q_i}-\sum_{j=1}^6 p_j)|\neq \emptyset$, hence $u(\mathcal{P})\subset \calQ_4$. Therefore there is an induced
map $\bar{u}:\mathcal{P}\dblq \op{Aut}(S')\dashrightarrow \calQ_4$ between $13$-dimensional varieties. Since every curve $(C, \eta)\in \calQ_4$ has a totally tangent conic and can be embedded in $S'$, it follows that any  $M\in W^1_4(C)$ with $h^0(C, M\otimes \eta)\geq 1$ appears in the way described above, which finishes the proof.
\end{proof}

Another distinguished codimension $2$ cycle in $\calR_6$ is the locus
$$\calQ_4':=\left\{(C, \eta)\in \calR_6: \eta\in W_2(C)-W_2(C)\right\}$$
of Prym curves $(C, \eta)$ for which $\varphi_{K_C\otimes \eta}$ fails to be very ample. Writing $\eta=\calO_C(a+b-p-q)$, with $a, b, p, q\in C$, then $M:=\calO_C(2a+2b)\in W^1_4(C)$ and the $2$-nodal image curve $\varphi_{K_C\otimes \eta}(C)$ lies on a pencil of quadrics in $\PP^4$, thus also on a singular quadric of type $(4, 6)$. We show however, that this quadric is \emph{not} the projectivized tangent cone of a quadratic singularity $L\in \Sing_{(C, \eta)}^{\mathrm{st}}(\Xi)$, hence points in $\calQ_4'$ do not constitute a component of $P^{-1}(H)$.
\begin{prop}\label{notveryample}
We have $\calQ_4'\nsubseteq \calU$. In particular, all singularities of the Prym theta divisor corresponding to a general point of $Q_4'$ are
ordinary double points, that is, $P(\calQ_4')\nsubseteq H$.
\end{prop}
\begin{proof} Note that $\calQ_4'$ is not contained in $\calU$, then use Proposition \ref{rk4sing}.
\end{proof}
\begin{prop}\label{q4par}
The locus $\calQ_4$ dominates via the Prym map the locus $H_1$, that is, $\overline{P(\calQ_4)}\supset H_1$.
\end{prop}
\begin{proof}
We start with a point $x_0=(\tau_0, z_0)\in \mathcal{S}'$, corresponding to a singular point $z_0\in \Theta_{\tau_0}$ such that $\rk H(x_0)\leq 4$ and $x_0$ is a general point of a component of $H-\theta_{\mathrm{null}}$. In particular $(A_{\tau_0}, \Theta_{\tau_0})$ can be chosen outside any  subvariety of $\calA_5$ having codimension at least $3$. Since each component of $\mathcal{S}'$ maps generically finite onto $N_0'$, we find a deformation $\left\{x_t=(\tau_t, z_t)\right\}_{t\in T}\subset \mathcal{S}'$, parameterized by an integral curve $T\ni 0$, such that for all $t\in T-\{0\}$, the corresponding theta divisor $\Theta_t$ has only a pair of singular points, that is, $\Sing(\Theta_t)=\{\pm z_t\}$. Since $P(\calQ)$ is dense in $N_0'$, after possibly shrinking $T$, we can find a family of triples $\{(C_t, \eta_t, L_t)\}_{t\in T}$, such that $(C_t, \eta_t)\in \calQ$
for all $t\in T$, while for $t\neq 0$ the line bundle $L_t\in V_3(C_t, \eta_t)$ corresponds to the singularity $z_t\in \Sing(\Xi_t)$. If we set $(C, L, \eta):=(C_0, L_0, \eta_0)$, by semicontinuity we obtain that $h^0(C, L)\geq 4$.  Since $\rk H(L)=\rk Q_{L}\leq 4$, it follows from Proposition \ref{rk4sing} that $L\in \Sing^{\mathrm{st}}_{(C, \eta)}(\Xi)\cap \Sing^{\mathrm{ex}}_{(C, \eta)}(\Xi)$, which implies that the Prym-canonical line bundle can be expressed as a sum of two pencils. Since $L$ is not a theta characteristic and $P(C, \eta)\notin \calJ_5$, we obtain that  the Prym-canonical bundle can be expressed as $K_C\otimes \eta=A\otimes A'$, where $A\in W^1_4(C)$. From Proposition \ref{notveryample} it follows that $K_C\otimes \eta$ can be assumed to be very ample, that is, $(C, \eta)\in \calQ_4$.
\end{proof}

\begin{cor}
$P(\calQ_4)$ is a unirational component of $H$, different from $\theta_{\mathrm{null}}^4$.
\end{cor}

\subsection{A parametrization of $\theta_{\mathrm{null}}^4$.}
Our aim is to find an explicit unirational parametrization of $\theta_{\mathrm{null}}^4$.
\begin{prop}\label{q5par}
 $\overline{P(\calQ_5)}=\theta_{\mathrm{null}}^4$, where the the closure is taken inside $\calA_5$.
\end{prop}
\begin{proof} This proof resembles that of Proposition \ref{q4par}. If $\phi:\calX_5\rightarrow \calA_5$ denotes the universal abelian variety, recall that we have showed that $\phi_*(\calS_{\mathrm{null}}\cap \calS')=\theta_{\mathrm{null}}^4$. Thus a point $(\tau, z)\in \calS_{\mathrm{null}}\cap \calS'$ corresponding to a general point $(A_{\tau}, \Theta_{\tau})$ of a component of $\theta_{\mathrm{null}}^4$ is a Prym variety $P(C, \eta)$, where $(C, \eta)\in \calQ\cap \calU$ is a Prym curve such that $z\in \Sing(\Theta_{\tau})$ corresponds to a singularity $L\in \Sing^{\mathrm{st}}_{(C, \eta)}(\Xi) \cap \Sing_{(C, \eta)}^{\mathrm{ex}}(\Xi)$. Then $L=f^*(\theta_1)$, where $\theta_1\in \Pic^5(C)$ is a vanishing theta-null. Since $h^0(\tilde{C}, L)=h^0(C, \theta_1)+h^0(C, \theta_1\otimes \eta)\geq 4$, we find that $\theta_2:=\theta_1\otimes \eta$ is another theta characteristic, that is, $(C, \eta)\in \calQ_5$. Therefore $\theta_{\mathrm{null}}^4\subseteq \overline{P(\calQ_5)}$. The reverse inclusion being obvious, we finish the proof.
\end{proof}

We can now complete the proof of Theorem \ref{q5}. We consider the smooth quadric $Q:=\PP^1\times \PP^1$ and the linear systems of rational curves
$$\PP^7_{1}:=\left|\calO_{\PP^1\times \PP^1}(3, 1)\right| \ \mbox{ and  } \PP^7_2:=\left|\calO_{\PP^1\times \PP^1}(1, 3)\right|.$$ Over $\PP^7_1\times \PP^7_2$ we define the $\PP^5$-bundle
$$\calU:=\left\{(R_1, R_2, \Gamma): R_i\in \PP^7_i \mbox{ for } i=1, 2, \ \ \Gamma\in |\mathcal{I}_{R_1\cdot R_2/Q}^2(5, 5)|\right\}.$$
There is an induced rational map $\psi:\calU \dashrightarrow \calQ_5$ given by
$$\psi(R_1, R_2, \Gamma):=\left(C, p_1^*\calO(1)\otimes p_2^*\calO(-1)\right)\in \calR_6,$$ where
$\nu:C\rightarrow \Gamma$ is the normalization map and $p_1, p_2:C\rightarrow \PP^1$ are the composition of $\nu$ with the two projections.

A general pair $(R_1, R_2)\in \PP^7_1\times \PP^7_2$ corresponds to smooth rational curves such that the intersection cycle
$R_1\cdot R_2=o_1+\cdots+o_{10}$ consists of distinct points. For any curve $\Gamma \in |\mathcal{I}^2_{R_1\cdot R_2}(5, 5)|$ we have
$R_1\cdot \Gamma=R_2\cdot \Gamma=2(o_1+\cdots+o_{10})$. Since $\nu^*:|\calI_{R_1\cdot R_2}(3, 3)|\rightarrow |K_C|$ is an isomorphism, it follows that both $p_1^*\calO(1)$ and $p_2^*\calO(1)$ are vanishing theta-nulls, hence $\psi(\calU)\subset \calQ_5$.

\begin{thm}\label{q5param}
The rational map $\psi:\calU \dashrightarrow \calQ_5$ is generically finite and dominant. In particular $\calQ_5$ (and thus $\tn^4=\overline{P(\calQ_5)}$) is unirational.
\end{thm}
\begin{proof}
We start with a point $(C, \theta_1, \theta_2)\in \calQ_5$ moving in a $13$-dimensional family. In particular, the image $\Gamma$ of the induced map $\varphi_{(\theta_1, \theta_2)}:C\rightarrow \PP^1\times \PP^1$ is nodal and we set $\Sing(\Gamma)=\{o_1, \ldots, o_{10}\}$.

We choose divisors $D, D'\in |\theta_1|$, corresponding to lines $\ell, \ell'\in |\calO_{Q}(1, 0)|$ such that $\nu^*(\Gamma\cdot \ell)=D$ and $\nu^*(\Gamma\cdot \ell')=D'$ respectively. Then $D+D'\in |K_C|$, and since the linear system $|\calI_{o_1+\cdots+o_{10}}(3, 3)|$ cuts out the canonical system on $C$, it follows that there exists a cubic curve $E\in |\calO_Q(3, 3)|$ such that
$$E\cdot \Gamma=D+D'+2\sum_{i=1}^{10} o_i.$$
By B\'ezout's Theorem, both $\ell$ and $\ell'$ must be components of $E$, that is, we can write $E=\ell+\ell'+R_1$, where
$R_1\in |\calO_{Q}(1, 3)|$ is such that $R_1\cdot \Gamma=2\sum_{i=1}^{10} o_i$. Switching the roles of $\theta_1$ and $\theta_2$, there exists
$R_2\in |\calO_{Q}(3, 1)|$ such that $R_2\cdot \Gamma=2\sum_{i=1}^{10} o_i$. It follows that $(R_1, R_2, \Gamma)\in \psi^{-1}\left((C, \theta_1\otimes \theta_2^{\vee})\right)$. The variety $\calU$ being a $\PP^5$-bundle over $\PP^7_1\times \PP^7_2$ is unirational, hence $\calQ_5$ is unirational as well, thus finishing the proof.
\end{proof}

\vskip 3pt
\noindent \small{{\bf{Acknowledgments:}} Research of the first author is supported by the Sonderforschungsbereich 647 ``Raum-Zeit-Materie" of the DFG.
The first author is grateful to INdAM and Universit\`a Roma Tre for providing financial support and a stimulating environment during a stay in Rome in which parts of this paper have been completed. Research of the second author was supported in part by National Science Foundation under the grant DMS-10-53313. Research of the third author is supported in part by Universit\`a ``La Sapienza'' under the grant C26A10NABK. Research of the fourth author is supported in part by the research project ``Geometry of Algebraic Varieties" of the Italian frame program PRIN-2008.}

\end{document}